\documentclass[final,onefignum,onetabnum]{siamltex1213} 

\usepackage{amsmath}
\usepackage{amsfonts}
\usepackage{amssymb}
\usepackage{graphicx}
\usepackage{xspace}
\usepackage{bm}
\usepackage{threeparttable}
\usepackage{subfigure}
\usepackage{afterpage}
\usepackage{soul}

\usepackage{hyperref}
\hypersetup{
  colorlinks=true,
  linkcolor=blue,
  citecolor=blue,
  urlcolor=blue
}

\newcommand{\fnc}[1]{\ensuremath{\mathcal{#1}}}
\newcommand{\mat}[1]{\ensuremath{\mathsf{#1}}}
\newcommand{\M}[0]{\mat{H}}
\newcommand{\Dx}[0]{\mat{D}_{x}}
\newcommand{\Dy}[0]{\mat{D}_{y}}
\newcommand{\Dz}[0]{\mat{D}_{z}}

\newcommand{\Qx}[0]{\mat{S}_{x}}
\newcommand{\Qy}[0]{\mat{S}_{y}}

\newcommand{\QAx}[0]{\mat{Q}_{x}}
\newcommand{\QAy}[0]{\mat{Q}_{y}}
\newcommand{\QAz}[0]{\mat{Q}_{z}}

\newcommand{\Ex}[0]{\mat{E}_{x}}
\newcommand{\Ey}[0]{\mat{E}_{y}}
\newcommand{\Ez}[0]{\mat{E}_{z}}
\newcommand{\Fx}[0]{\mat{F}_{x}}
\newcommand{\Fy}[0]{\mat{F}_{y}}
\newcommand{\Gx}[0]{\mat{G}_{x}}
\newcommand{\Gy}[0]{\mat{G}_{y}}
\newcommand{\tEx}[0]{\tilde{\mat{E}}_{x}}
\newcommand{\tEy}[0]{\tilde{\mat{E}}_{y}}
\newcommand{\V}[0]{\mat{V}}
\newcommand{\tV}[0]{\tilde{\mat{V}}}
\newcommand{\Vx}[0]{\mat{V}_{x}}
\newcommand{\tVx}[0]{\tilde{\mat{V}}_{x}}
\newcommand{\W}[0]{\mat{W}}
\newcommand{\Wx}[0]{\mat{W}_{x}}
\renewcommand{\P}[0]{\mat{P}}
\newcommand{\Px}[0]{\mat{P}_{x}}
\newcommand{\ax}[0]{a_{x}}
\newcommand{\bx}[0]{b_{x}}
\newcommand{\ay}[0]{a_{y}}
\newcommand{\by}[0]{b_{y}}
\newcommand{\nmin}[1]{n^{*}_{#1}}
\newcommand{\nglob}[0]{\tilde{n}}
\newcommand{\iglob}[0]{\tilde{\imath}}
\newcommand{\jglob}[0]{\tilde{\jmath}}

\newcommand{\Tr}{\ensuremath{^{\mr{T}}}}
\newcommand{\invTr}{\ensuremath{^{-\mr{T}}}}
\newcommand{\mr}[1]{\ensuremath{\mathrm{#1}}}

\newtheorem{remark}{Remark}

\newcommand{\ignore}[1]{}
\newcommand{\etal}[0]{{\em et~al.\@}\xspace}
\newcommand{\eg}[0]{{e.g.\@}\xspace}
\newcommand{\ie}[0]{{i.e.\@}\xspace}

\newcommand{\fpk}{\fnc{P}_{k}}
\newcommand{\fpm}{\fnc{P}_{m}}
\newcommand{\fpl}{\fnc{P}_{l}}
\newcommand{\pk}{\bm{p}_{k}}
\newcommand{\pM}{\bm{p}_{m}}

\newcommand{\dxfpk}{\frac{\partial\fnc{P}_{k}}{\partial x}}
\newcommand{\dxfpm}{\frac{\partial\fnc{P}_{m}}{\partial x}}

\newcommand{\dxpk}{\bm{p}_{k}'}
\newcommand{\dxpm}{\bm{p}_{m}'}

\newcommand{\tauEx}{\tau} 
\title{Multi-dimensional Summation-By-Parts Operators:\\ General Theory and
  Application to Simplex Elements\thanks{This work was supported by Rensselaer
    Polytechnic Institute and the Natural Sciences and Engineering Research
    Council (NSERC) of Canada}}

\author{Jason~E. Hicken\footnotemark[2]\ \footnotemark[5]\ 
\and David~C.~Del~Rey~Fern\'andez\footnotemark[3]\ \footnotemark[6]
\and David~W.~Zingg\footnotemark[4]\ \footnotemark[6]}

\begin{document}

\maketitle
\slugger{sisc}{xxxx}{xx}{x}{x--x}

\renewcommand{\thefootnote}{\fnsymbol{footnote}}

\footnotetext[2]{Assistant Professor ({\tt hickej2@rpi.edu})}
\footnotetext[3]{Postdoctoral Fellow ({\tt dcdelrey@gmail.com})}
\footnotetext[4]{Professor and Director ({\tt dwz@oddjob.utias.utoronto.ca})}
\footnotetext[5]{Department of Mechanical, Aerospace, and Nuclear Engineering,
  Rensselaer Polytechnic Institute, Troy, New York, United States}
\footnotetext[6]{Institute for Aerospace Studies, University of Toronto,
  Toronto, Ontario, M3H 5T6, Canada}

\renewcommand{\thefootnote}{\arabic{footnote}}

\begin{abstract}
Summation-by-parts (SBP) finite-difference discretizations share many attractive
properties with Galerkin finite-element methods (FEMs), including time stability
and superconvergent functionals; however, unlike FEMs, SBP operators are not
completely determined by a basis, so the potential exists to tailor SBP
operators to meet different objectives.  To date, application of high-order SBP
discretizations to multiple dimensions has been limited to tensor product
domains.  This paper presents a definition for multi-dimensional SBP
finite-difference operators that is a natural extension of one-dimensional SBP
operators.  Theoretical implications of the definition are investigated for the
special case of a diagonal norm (mass) matrix.  In particular, a diagonal-norm
SBP operator exists on a given domain if and only if there is a cubature rule
with positive weights on that domain and the polynomial-basis matrix has full
rank when evaluated at the cubature nodes.  Appropriate
simultaneous-approximation terms are developed to impose boundary conditions
weakly, and the resulting discretizations are shown to be time stable.  Concrete
examples of multi-dimensional SBP operators are constructed for the triangle and
tetrahedron; similarities and differences with spectral-element and
spectral-difference methods are discussed.  An assembly process is described
that builds diagonal-norm SBP operators on a global domain from element-level
operators.  Numerical results of linear advection on a doubly periodic domain
demonstrate the accuracy and time stability of the simplex operators.
\end{abstract}

\begin{keywords}
summation-by-parts, finite-difference method, unstructured grid,
spectral-element method, spectral-difference method, mimetic discretization
\end{keywords}

\begin{AMS}
65N06, 65M60, 65N12
\end{AMS}

\pagestyle{myheadings}
\thispagestyle{plain}
\markboth{J. E. HICKEN, D. C. DEL~REY~FERN\'ANDEZ, AND D. W. ZINGG}{MULTI-DIMENSIONAL SBP OPERATORS}

\section{Introduction}

Summation-by-parts (SBP) operators are high-order finite-difference schemes that
mimic the symmetry properties of the differential operators they
approximate~\cite{Kreiss1974}.  Respecting such symmetries has important
implications; in particular, they enable SBP discretizations that are both
time stable and high-order accurate~\cite{Carpenter1999, Yee2002,
  Nordstrom2006}\ignore{other references that focus on time-stable SBP-SAT
  methods}, properties that are essential for robust, long-time simulations of
turbulent flows~\cite{Morinishi1998, Yee2000}.

Most existing SBP operators are one-dimensional~\cite{Strand1994, Mattsson2004b,
  Svard2004b, Mattsson2012}\ignore{GSBP, other references that focus on SBP
  operators themselves} and are applied to multi-dimensional problems using a
multi-block tensor-product formulation~\cite{Svard2005, Hicken2008,
  Nordstrom2009}\ignore{other MB examples}.  Like other tensor-product methods,
the restriction to multi-block grids complicates mesh generation and adaptation,
and it limits the geometric complexity that can be considered in practice.

The limitations of the tensor-product formulation motivate our interest in
generalizing SBP operators to unstructured grids.  There are two ways this
generalization has been pursued in the literature: 1) construct global SBP
operators on an arbitrary distribution of nodes, or; 2) construct SBP operators
on reference elements and assemble a global discretization by coupling these
smaller elements.

The first approach is appealing conceptually, and it is certainly viable for
second-order accurate SBP schemes.  For example, Nordstr\"om
\etal~\cite{Nordstrom2003} showed that the vertex-centered second-order-accurate
finite-volume scheme\footnote{On simplices, the vertex-centered finite-volume
  scheme is equivalent to a mass-lumped $p=1$ finite-element discretization} has
a multi-dimensional SBP property, even on unstructured grids; however, the first
approach presents challenges when constructing high-order operators.  Kitson
\etal~\cite{Kitson2003} showed that, for a given stencil width and design
accuracy, there exist grids for which no stable, diagonal-norm SBP operator
exists.  Thus, building stable high-order SBP operators on arbitrary node
distributions may require unacceptably large stencils.  When SBP operators do
exist for a given node distribution, they must be determined globally by solving
a system of equations, in general.  The global nature of these SBP operators is
exemplified in the mesh-free framework of Chiu \etal \cite{Chiu2012, Chiu2011}.

The second approach --- constructing SBP operators on reference elements and
using these to build the global discretization --- is more common and presents
fewer difficulties.  The primary challenge here is to extend the one-dimensional
SBP operators of Kreiss and Scherer \cite{Kreiss1974} to a broader set of
operators and domains.  The existence of such operators, at least in the
dense-norm case\footnote{In this paper, norm matrix is synonymous with mass
  matrix.}, was established by Carpenter and Gottlieb \cite{Carpenter1996}.
They proved that operators with the SBP property can be constructed from the
Lagrangian interpolant on nearly arbitrary nodal distributions, which is
practically feasible on reference elements with relatively few nodes.  More
recently, Gassner \cite{Gassner2013} showed that the discontinuous
spectral-element method is equivalent to a diagonal-norm SBP discretization when
the Legendre-Gauss-Lobatto nodes are used with a lumped mass matrix.

Of particular relevance to the present work is the extension of the SBP concept
by Del Rey Fernandez \etal \cite{DCDRF2014}.  They introduced a generalized
summation-by-parts (GSBP) definition for arbitrary node distributions on
one-dimensional elements, and these ideas helped shape the definition of SBP
operators presented herein.

Our first objective in the present work is to develop a suitable definition for
multi-dimensional SBP operators on arbitrary grids and to characterize the
resulting operators theoretically.  We note that the discrete-derivative
operator presented in~\cite{Chiu2012} is a possible candidate for defining
(diagonal-norm) multi-dimensional SBP operators; however, it lacks properties of
conventional SBP operators that we would like to retain, such as the accuracy of
the discrete divergence theorem~\cite{Hicken2013}.

Our second objective is to provide a concrete example of multi-dimensional
\linebreak diagonal-norm SBP operators on non-tensor-product domains.  We focus
on diagonal-norm operators, because they are better suited to discretizations
that conserve non-quadratic invariants~\cite{Kitson2003}; they are also more
attractive than dense norms for explicit time-marching schemes.  We construct
diagonal-norm SBP operators for triangular and tetrahedral elements.  The
resulting operators are similar to those used in the nodal
triangular-spectral-element method~\cite{Cohen2001, Mulder2001, Giraldo2006}.
Unlike the spectral-element method based on cubature points, the SBP method is
not completely specified by a polynomial basis; we use the resulting freedom to
enforce the summation-by-parts property, which leads to provably time-stable
schemes.

The remaining paper is structured as follows.  Section~\ref{sec:preliminaries}
presents notation and the proposed definition for multi-dimensional SBP
operators.  We study the theoretical implications of the proposed definition in
Section~\ref{sec:theory}.  We then describe, in Section~\ref{sec:construct}, how
to construct diagonal-norm SBP operators for the triangle and tetrahedron.
Section~\ref{sec:construct} also establishes that SBP operators on subdomains
can be assembled into SBP operators on the global domain.  Results of applying
the triangular SBP operators to the linear advection equation are presented in
Section~\ref{sec:results}.  Conclusions are given in Section~\ref{sec:conclude}.

\section{Preliminaries}\label{sec:preliminaries} 

To make the presentation concise, we concentrate on multi-dimensional SBP
operators in two dimensions; the extension to higher dimensions follows in a
straightforward manner.  Furthermore, we present definitions and
theorems for operators in the $x$ coordinate direction only; the corresponding definitions and theorems for the $y$ coordinate direction follow directly from those in the $x$ direction.

\subsection{Notation}\label{sec:notation}

We consider discretized derivative operators defined on a set of $n$ nodes, $S=\left\{(x_{i},y_{i})\right\}_{i=1}^{n}$.  Capital letters with script type are used to denote continuous functions.  For example, $\fnc{U}(x) \in L^{2}(\Omega)$ denotes a square-integrable function on the domain $\Omega$.  We use lower-case bold font to denote the restriction of functions to the nodes. Thus, the restriction of $\fnc{U}$ to $S$ is given by
\begin{equation*}
\bm{u} = \left[\fnc{U}(x_{1},y_{1}),\dots,\fnc{U}(x_{n},y_{n})\right]\Tr.
\end{equation*}

\ignore{
Following this convention, the nodes themselves will often be represented by the two vectors $\bm{x}=\left[x_{1},\dots,x_{n}\right]\Tr$ and $\bm{y}=\left[y_{1},\dots,y_{n}\right]\Tr$.  More generally, the restriction of monomials to $S$ is represented by $\bm{x}^{j} = \left[x_{1}^{j},\dots,x_{n}^{j}\right]\Tr$ and  $\bm{y}^{j} = \left[y_{1}^{j},\dots,y_{n}^{j}\right]\Tr$, with the convention that $\bm{x}^{j}=\bm{y}^{j}=\bm{0}$ if $j<0$.  We use the element-wise Hadamard product, denoted $\circ$, to represent the product of functions restricted to the nodes.  For example, the restriction of $x^{a}y^{b}$ to $S$ is given by $\bm{x}^{a}\circ\bm{y}^{b}$.
}

Several theorems and proofs make use of the monomial basis.  For two spatial
variables, the size of the polynomial basis of total degree $p$ is
\begin{equation*}
  \nmin{p} \equiv \frac{(p+1)(p+2)}{2}.
\end{equation*}
More generally, $\nmin{p} = \binom{p+d}{d}$, where $d$ is the spatial dimension.
We use the following single-subscript notation for monomial basis functions:
\begin{equation*}
  \fpk(x,y) \equiv x^{i}y^{j-i}, \qquad
  k = j(j+1)/2 +i+1, \quad 
  \forall\; j \in \{ 0,1,\ldots,p\}, \quad i \in \{0,1,\ldots,j\}.
\end{equation*}
We will frequently evaluate $\fpk$ and $\partial \fpk/\partial x$ at the nodes
$S$, so we introduce the notation 
\begin{align*}
  \pk &\equiv \left[ \fpk(x_{1},y_{1}),\dots,\fpk(x_{n},y_{n})\right]\Tr,\\
\text{and}\qquad
  \dxpk &\equiv \left[ \frac{\partial \fpk}{\partial x}(x_{1},y_{1}),\dots,
    \frac{\partial \fpk}{\partial x}(x_{n},y_{n})\right]\Tr.
\end{align*}

\ignore{JEH: I did not see where we needed to use the $_h$ notation.

Vectors with a subscript $h$, for example $\bm{u}_{h}\in \mathbb{R}^{n}$, represent the solution to a system of discrete or semi-discrete equations.
}

Finally, matrices are represented using capital letters with sans-serif font;
for example, the first-derivative operators with respect to $x$ and $y$ are
represented by the matrices $\Dx$ and $\Dy$, respectively.  Entries of a matrix
are indicated with subscripts, and we follow
Matlab\textsuperscript{\textregistered}-like notation when referencing
submatrices.  For example, $\mat{A}_{:,j}$ denotes the $j^{\text{th}}$ column of
matrix $\mat{A}$, and $\mat{A}_{:,1:k}$ denotes its first $k$ columns.

\ignore{ I removed the this definition, which we do not need to include.

The $L_{2}$  inner product and norm on the domain $\Omega$ are defined as
\begin{equation}
\begin{array}{lr}
(\fnc{U},\fnc{V}) = \displaystyle\int_{\Omega}\fnc{U}\fnc{V}\mr{d}\Omega,&||\fnc{U}||^{2} = \displaystyle\int_{\Omega}\fnc{U}^{2}\mr{d}\Omega.
\end{array}
\end{equation}
A discrete inner product and norm have the form
\begin{equation}\label{SBPnorm1}
\begin{array}{lr}
(\bm{u},\bm{v})_{\M} = \bm{u}\Tr\M\bm{v},&||\bm{u}||^{2}_{\M} = \bm{u}\Tr\M\bm{u},
\end{array}
\end{equation}
where $\M$ must be symmetric and positive-definite.
}

\subsection{Multi-dimensional SBP operator definition}

We propose the following definition for $\Dx$, the SBP first-derivative operator
with respect to $x$.  An analogous definition holds for $\Dy$ and, in
three-dimensions, $\Dz$.  Definition \ref{DEFSBPgen2} is a natural extension of
the definition of GSBP operators proposed in \cite{DCDRF2014}, which itself
extends the classical SBP operators introduced by Kreiss and Scherer
\cite{Kreiss1974}.

\begin{definition}\label{DEFSBPgen2}
  {\bf Two-dimensional summation-by-parts operator:} Consider an open and bounded domain $\Omega\subset\mathbb{R}^{2}$ with a piecewise-smooth boundary $\Gamma$.  The matrix $\Dx$ is a degree $p$ SBP approximation to the first derivative $\frac{\partial}{\partial x}$ on the nodes $S=\left\{(x_{i},y_{i})\right\}_{i=1}^{n}$ if
  \begin{remunerate}
  \item $\Dx\pk = \dxpk,\qquad \forall\; k \in \{ 1,2,\ldots,\nmin{p} \}$; 
    \label{sbp:accuracy}
  \item $\Dx = \M^{-1}\Qx$, where $\M$ is symmetric positive-definite, and; \label{sbp:H}
  \item $\Qx = \QAx + \frac{1}{2}\Ex$, where $\QAx\Tr=-\QAx$, $\Ex\Tr=\Ex$, and
    $\Ex$ satisfies
    \begin{equation*}
      \pk\Tr\Ex\pM=\displaystyle\oint_{\Gamma}\fpk\fpm n_{x}
      \mr{d}\Gamma,\qquad \forall\; k,m \in \{ 1,2,\ldots,\nmin{\tauEx} \},
    \end{equation*}
    \label{sbp:Ex} 
  \end{remunerate}
  where $\tauEx\ge p$ and $n_{x}$ is the $x$ component of
  $\bm{n}=\left[n_{x},n_{y}\right]\Tr$, the outward pointing unit normal on
  $\Gamma$.
\end{definition}

Before studying the implications of Definition~\ref{DEFSBPgen2} in
Section~\ref{sec:theory}, it is worthwhile to motivate and elaborate on the
three properties in the definition.

Property~\ref{sbp:accuracy} ensures that $\Dx$ is an accurate approximation to
the first partial derivative with respect to $x$.  The operator must be exact
for polynomials of total degree less than or equal to $p$, so at least
$\nmin{p}$ nodes are necessary to satisfy property~\ref{sbp:accuracy}.

\begin{remark}
We emphasize that a polynomial basis is not used to define the solution in SBP
methods, in contrast with the piecewise polynomial expansions found in
finite-element methods.  We adopt the monomial basis only to define the accuracy
conditions concisely and avoid cumbersome Taylor-series expansions.
\end{remark}

The matrix $\M$ must be symmetric positive-definite to guarantee stability:
without property~\ref{sbp:H}, the discrete ``energy'', $\bm{u}\Tr \M \bm{u}$,
could be negative when $\bm{u}\Tr\bm{u} > 0$, and vice versa.  The so-called
norm matrix $\M$ can be interpreted as a mass matrix, \ie
\begin{equation*}
  \M_{i,j} = \int_{\Omega} \phi_{i}(x,y) \phi_{j}(x,y) \mr{d} \Omega,
\end{equation*}
but it is important to emphasize that SBP operators are finite-difference
operators, and there is no (known) closed-form expression for an SBP nodal basis
$\{\phi_{i}\}_{i=1}^{n}$, in general.  In the diagonal norm case, we shall show
that another interpretation of $\M$ is as a cubature rule.

Property~\ref{sbp:Ex} is needed to mimic integration by parts (IBP).  Recall
that the IBP formula for the $x$ derivative is
\begin{equation*}
\int_{\Omega}\fnc{V}\frac{\partial\fnc{U}}{\partial x}\mr{d}\Omega 
+ \int_{\Omega}\fnc{U}\frac{\partial\fnc{V}}{\partial x}\mr{d}\Omega
= \oint_{\Gamma}\fnc{V}\fnc{U}n_{x}\mr{d}\Gamma.
\end{equation*}
The discrete version of the IBP formula, which follows from
property~\ref{sbp:Ex}, is
\begin{equation*}
\bm{v}\Tr\M\Dx\bm{u} + \bm{u}\Tr\M\Dx\bm{v} =\bm{v}\Tr\Ex\bm{u},
\qquad\forall\;\bm{v},\bm{u}\in\mathbb{R}^{n}.
\end{equation*}
There is a one-to-one correspondence between each term in the IBP formula and
its SBP proxy.  For example, it is clear from property~\ref{sbp:Ex} that
$\bm{v}\Tr \Ex\bm{u}$ approximates the surface integral in IBP to order $\tau$.
Moreover, in Section~\ref{sec:theory} we show that diagonal-norm SBP operators
also approximate the left-hand side of IBP.

\section{Analysis of diagonal-norm, multi-dimensional, summation-by-parts operators}\label{sec:theory}
In this section, we determine the implications of Definition \ref{DEFSBPgen2} on
the constituent matrices of a multi-dimensional SBP operator and whether or not
such operators exist.  We also investigate the time stability of discretizations
based on multi-dimensional SBP operators.  The focus is on diagonal-norm
operators; however, the ideas presented here can be extended to dense-norm
operators, \ie where the matrix $\M$ is not diagonal.

The following lemma will prove useful in the sequel.  It follows immediately
from properties \ref{sbp:accuracy} and \ref{sbp:Ex}, so we state it without
proof.

\begin{lemma}[compatibility] Let $\Dx = \M^{-1}(\QAx + \frac{1}{2} \Ex)$ be an SBP operator of degree $p$.  Then we have the following set of relations:
\begin{equation}\label{MSBP:compatx}
  \pM\Tr\M\dxpk+\pk\Tr\M\dxpm=\pM\Tr\Ex\pk,\qquad
  \forall\; k,m \in \{ 1,2,\ldots,\nmin{p} \}.
\end{equation}
\end{lemma}

We refer to \eqref{MSBP:compatx} as the compatibility equations for the $x$ derivative; $\M$ must simultaneously satisfy analogous relations for $\Ey$. The relation between $\M$ and $\Ex$ was first derived by Kreiss and Scherer \cite{Kreiss1974} and Strand \cite{Strand1994} to construct a theory for one-dimensional classical finite-difference-SBP operators. Furthermore, Del Rey Fern\'andez \etal \cite{DCDRF2014} have used these relations to extend the theory of such operators to a broader set.  What is presented in this paper is a natural extension of those works to multi-dimensional operators, and the derivation of \eqref{MSBP:compatx} follows in a straightforward manner from any of the mentioned works.

Our first use of the compatibility equations is to prove that, in the
diagonal-norm case, the multi-dimensional SBP definition conceals a cubature rule
with positive weights.
\begin{theorem}\label{L2diagH}
Let $\Dx = \M^{-1}\Qx$ be a degree $p$, diagonal-norm, multi-dimensional SBP
operator on the domain $\Omega$.  Then the nodes $S =
\{(x_{i},y_{i})\}_{i=1}^{n}$ and the diagonal entries of $\M$ form a degree
$2p-1$ cubature rule on $\Omega$.
\end{theorem}
\begin{proof}
Using property~\ref{sbp:Ex} of Definition~\ref{DEFSBPgen2}, the compatibility
equations become
\begin{equation*}
\sum\limits_{j=1}^{n}\M_{j,j}\left[\fpm\dxfpk + \fpk\dxfpm\right]_{\left(x_{j},y_{j}\right)}=\oint\limits_{\Gamma}\fpm\fpk n_{x}\mr{d}\Gamma,\qquad
\forall\; k,m \in \{ 1,2,\dots,\nmin{p} \}.
\end{equation*}
Using the chain rule on the left and integration by parts on the right results
in
\begin{equation}\label{proofH}
\sum\limits_{j=1}^{n}\M_{j,j} \left.\frac{\partial\fpm\fpk}{\partial x}\right|_{\left(x_{j},y_{j}\right)}=\int\limits_{\Omega}\frac{\partial\fpm\fpk}{\partial x}\mr{d}\Omega,\qquad  \forall\; k,m \in \{ 1,2,\dots,\nmin{p} \}.
\end{equation}
Since $\fpk$ and $\fpm$ are monomials of degree at most $p$, it follows that
$\partial\left(\fpm\fpl\right)/\partial x$ is a scaled monomial of degree at
most $2p-1$; thus, by considering all of the combinations of $k$ and $m$,
\eqref{proofH} implies
\begin{equation*}
\sum\limits_{j=1}^{n} \M_{j,j} \fpk\left(x_{j},y_{j}\right)=\int\limits_{\Omega}\fpk\mr{d}\Omega,\quad \forall\; k \in \{ 1,2,\dots,\nmin{2p-1}\},
\end{equation*}
which are the conditions for a degree $2p-1$ cubature.
\ignore{
This result follows in an analogous fashion to the one-dimensional result;
see~\cite[\S 4.1]{DCDRF2014}. Briefly, one starts with the compatibility
equations for the $x$ coordinate \eqref{MSBP:compatx} and proves the result. It
then follows that the compatibility equations for the $y$ coordinate are also
satisfied.  \qquad
}
\qquad\end{proof}

We now prove one of our central theoretical results, relating the existence of a
diagonal-norm SBP operator to the existence of a cubature rule with positive
weights.
\begin{theorem}\label{ExistDiagH}
  Consider the node set $S=\left\{(x_{i},y_{i})\right\}_{i=1}^{n}$ with $n \ge
  \nmin{p}$ nodes, and define the generalized Vandermonde matrix $\mat{V} \in
  \mathbb{R}^{n\times \nmin{p}}$ whose columns are the monomial basis
  evaluated at the nodes;
\begin{equation*}
  \mat{V}_{:,k} =\pk, \qquad \forall\; k \in \{1,2,\dots,\nmin{p}\}
\end{equation*}
Assume that the columns of $\mat{V}$ are linearly independent.  Then the
existence of a cubature rule of degree at least $2p-1$ with positive weights on
$S$ is necessary and sufficient for the existence of degree $p$ diagonal-norm
SBP operators approximating the first derivatives $\frac{\partial}{\partial x}$
and $\frac{\partial}{\partial y}$ on the node set $S$.
\end{theorem}

\begin{proof}
  The necessary condition on $\M$ follows immediately from Theorem \ref{L2diagH}.  To
  prove sufficiency, we must show that, given a cubature rule, we can construct
  an operator that satisfies properties \ref{sbp:accuracy}--\ref{sbp:Ex} of
  Definition~\ref{DEFSBPgen2} on the same node set as the cubature rule.

  Before proceeding, we introduce some matrices that facilitate the proof.  Let
  $\mat{V}_{x} \in \mathbb{R}^{n\times \nmin{p}}$ be the matrix whose columns
  are the $x$ derivatives of the monomial-basis:
  \begin{equation*}
    \left(\mat{V}_{x}\right)_{:,k} =\dxpk,
    \qquad \forall\; k \in \{1,2,\dots,\nmin{p}\}.
  \end{equation*}
  We construct an invertible matrix $\tV \in \mathbb{R}^{n\times n}$ by
  appending a set of vectors, $\W \in \mathbb{R}^{n\times (n-\nmin{p})}$, that
  are linearly independent amongst themselves and to the vectors in $\V$:
  \begin{equation*}
    \tV \equiv \begin{bmatrix} \V & \W \end{bmatrix}.
  \end{equation*}
  Similarly, we define
  \begin{equation*}
    \tVx \equiv \begin{bmatrix} \Vx & \Wx \end{bmatrix},
  \end{equation*}
  where $\Wx \in \mathbb{R}^{n\times (n-\nmin{p})}$ is matrix that will be specified later.
  Below, we use the degrees of freedom in $\Wx$ to satisfy the SBP definition.

  Let $\M$ be the diagonal matrix whose entries are the cubature weights ordered
  consistently with the cubature node set $S$.  Since the cubature weights are
  positive, property~\ref{sbp:H} is satisfied.

  Next, we use the cubature to construct a suitable $\Ex$.  Using $\mat{V}$
  and $\mat{V}_{x}$, we define the symmetric matrix
  \begin{equation}\label{tExdeff}
    \tEx \equiv \mat{V}\Tr \M \mat{V}_{x} + \mat{V}_{x}\Tr \M \mat{V}.
  \end{equation}
  Since $\mat{V}$ and $\mat{V}_{x}$ are polynomials of degree $p$ and $p-1$,
  respectively, evaluated at the nodes, the cubature is exact for the right-hand
  side of \eqref{tExdeff}:
  \begin{equation}
    \left(\tEx\right)_{k,m} =
    \int_{\Omega}\dxfpk\fpm \mr{d}\Omega
    + \int_{\Omega} \fpk\dxfpm \mr{d}\Omega = \oint_{\Gamma} \fpk\fpm n_{x}\mr{d}\Gamma,
    \label{eq:tildeEx}
  \end{equation}
  $\forall\; k,m \in \{ 1,2,\dots,\nmin{p} \}$. 
  Now we can define the boundary operator
  \begin{equation*}
    \Ex \equiv \tV\invTr 
    \begin{bmatrix}
      \tEx & \Fx\Tr \\
      \Fx & \Gx 
    \end{bmatrix} \tV^{-1},
  \end{equation*}
  where $\Fx \in \mathbb{R}^{(n-\nmin{p})\times \nmin{p}}$ and $\Gx =
  \Gx\Tr \in \mathbb{R}^{(n-\nmin{p})\times (n-\nmin{p})}$.  It follows from
  this definition that $\Ex$ is symmetric.  Moreover, together with
  \eqref{eq:tildeEx}, this definition implies
  \begin{equation*}
    \left( \V\Tr \Ex \V \right)_{k,m} = \left( \tEx \right)_{k,m}
    = \oint_{\Gamma} \fpk\fpm n_{x}\mr{d}\Gamma,\qquad
    \forall\; k,m \in \{1,2,\dots,\nmin{p} \},
  \end{equation*}
  so $\Ex$ satisfies the accuracy condition of property~\ref{sbp:Ex}.  

  Finally, let
  \begin{equation}
    \QAx \equiv \M \tVx \tV^{-1} - \frac{1}{2} \Ex.\label{eq:Qx}
  \end{equation}
  The accuracy conditions, which are equivalent to showing $\Dx \V = \Vx$,
  follow immediately from this definition of $\QAx$:
  \begin{equation*}
    \Dx \V = \M^{-1}\left( \QAx + \frac{1}{2} \Ex \right) \V 
    = \M^{-1} \left( \M \tVx \tV^{-1} \right) \V = \Vx;
  \end{equation*}
  thus, property~\ref{sbp:accuracy} is satisfied.

  Our remaining task is to show that $\QAx$ can be constructed to be
  antisymmetric.  If we can show that
  \begin{equation*}
    \tV\Tr \QAx \tV =
    \begin{bmatrix}
      \V\Tr \QAx \V & \V\Tr \QAx \W \\
      \W\Tr \QAx \V & \W\Tr \QAx \W 
    \end{bmatrix}
  \end{equation*}
  can be made antisymmetric, then the result will follow for $\QAx$.  Consider
  the first block in the $2\times 2$ block matrix above, \ie
  \begin{equation*}
    \V\Tr \QAx \V = \V\Tr \M \Vx - \frac{1}{2} \V\Tr \Ex \V.
  \end{equation*}
  Adding this block to its transpose, we find
  \begin{equation}\label{eq:comp_mat}
    \V\Tr \QAx \V + \V\Tr \QAx \Tr \V = 
    \V\Tr \M \Vx + \Vx\Tr \M \V - \V\Tr \Ex \V,
  \end{equation}
  where we have used the symmetry of $\Ex$.  The right-hand side of
  \eqref{eq:comp_mat} is the matrix form of the (rearranged) compatibility
  equations \eqref{MSBP:compatx}.  Thus, $\V\Tr \QAx \V + \V\Tr \QAx\Tr \V =
  \mat{0}$, proving that the first block is antisymmetric.  For the remaining
  three blocks, antisymmetry requires
  \begin{equation*}
    \left( \V\Tr \QAx \W \right)\Tr = -\W\Tr \QAx \V,
    \qquad\text{and}\qquad
    \W\Tr \QAx \W = -\W\Tr \QAx\Tr \W.
  \end{equation*}
  Substituting $\QAx$ and $\Ex$ and simplifying, we obtain the following
  equations:
  \begin{equation}\label{eq:skew_conditions}
    \Wx\Tr \M \V + \W\Tr \M \Vx = \Fx
    \qquad\text{and}\qquad
    \W\Tr \M \Wx + \Wx\Tr \M \W = \Gx.
  \end{equation}
  The two matrix equations above constitute $n(n- \nmin{p})$ scalar equations.
  We are free to choose $\Wx$, $\Gx$, and $\Fx$, so the matrix equations are
  underdetermined ($\Wx$ alone has $n(n-\nmin{p})$ entries).  To prove existence
  of an SBP operator we need only find one solution; for example, take $\Wx =
  \mat{0}$, $\Gx=\mat{0}$, and $\Fx = \W\Tr \M \Vx$.  Thus, the equations can be
  satisfied to ensure the antisymmetry of $\QAx$.
\qquad\end{proof}

\begin{remark}
  In general, there are infinitely many operators associated with a given
  cubature rule that satisfy Definition~\ref{DEFSBPgen2}.  For example, the
  proof of Theorem~\ref{ExistDiagH} only considered one way to solve
  \eqref{eq:skew_conditions}.  Another way to satisfy these conditions is to set
  $\Fx = \mat{0}$ and then solve
  \begin{equation*}
    \left(\V\Tr \M\right) \Wx = - \Vx\Tr \M \W
  \end{equation*}
  for $\Wx$ by finding the minimum Frobenius-norm solution.  $\Gx$ can then be
  computed directly from the second equation, \eqref{eq:skew_conditions}.
\end{remark}

The following theorem characterizes the matrices $\Qx$ and $\QAx$ in terms of
the bilinear forms that they approximate.  The theorem is useful when
discretizing the weak form of a PDE, rather than the strong form, using SBP
finite-difference operators.
\begin{theorem}\label{QxDiag}
Let $\Dx=\M^{-1}\Qx=\M^{-1}\left(\QAx+\frac{1}{2}\Ex\right)$ be a diagonal-norm SBP operator of degree $p$ on the domain $\Omega$.  Then
\begin{align}
\pk\Tr \Qx \pM &= \int_{\Omega} \fpk \dxfpm \mr{d} \Omega, \label{bilinearQx}
\qquad\forall\; k+m \leq \min\left(\nmin{2p}, \nmin{\tau}\right)
 \\
\pk\Tr \QAx \pM &= \int_{\Omega} \fpk \dxfpm \mr{d}\Omega
- \frac{1}{2}\oint_{\Gamma} \fpk \fpm n_{x} \mr{d}\Gamma,
\qquad\forall\; k+m \leq \min\left(\nmin{2p}, \nmin{\tau}\right).
 \label{bilinearQAx}
\end{align}
\end{theorem}
\begin{proof}
The proof of \eqref{bilinearQx} is analogous to the proof in \cite{Hicken2013}
for one-dimensional classical finite-difference-SBP operators.  With \eqref{bilinearQx} established, we can substitute $\Qx=\QAx+\frac{1}{2}\Ex$ and rearrange to obtain
\begin{align*} 
  \pk\Tr\QAx\pM&=\displaystyle\int_{\Omega}\fpk\dxfpm \mr{d}\Omega
  -\frac{1}{2}\pk\Tr\Ex\pM, \\
  &= \displaystyle\int_{\Omega}\fpk\dxfpm\mr{d}\Omega
  -\frac{1}{2}\oint_{\Gamma}\fpk\fpm n_{x} \mr{d}\Gamma,\qquad
  \forall\;k+m\leq \min\left(\nmin{2p},\nmin{\tauEx}\right),
\end{align*}
where we have used the accuracy property of $\Ex$ to get the desired result.
\qquad\end{proof}

\subsection{Stability Analysis}

We conclude our analysis of diagonal-norm multi-dimensional SBP operators by
investigating the stability of an SBP semi-discretization of the
constant-coefficient advection equation
\begin{equation}
  \begin{aligned}
    &\frac{\partial \fnc{U}}{\partial t} + \beta_{x} \frac{\partial \fnc{U}}{\partial x} +
    \beta_{y} \frac{\partial \fnc{U}}{\partial y} = 0, \qquad \forall\; (x,y) \in \Omega, \\
    &\fnc{U}(x,y,t) = \fnc{U}_{\text{bc}}(x,y,t),\qquad \forall\; (x,y) \in \Gamma_{-}, 
\end{aligned}\label{eq:linear_advect}
\end{equation}
where $\Gamma^{-} = \{ (x,y) \in \Gamma \;|\; \beta_{x} n_{x} + \beta_{y} n_{y}
< 0 \}$ is the inflow boundary and $\Gamma^{+} = \Gamma \setminus \Gamma^{-}$ is
the outflow boundary.

Let $\Dx$ and $\Dy$ be SBP operators and let $\Ex$ and $\Ey$ be their
corresponding boundary operators.  In order to impose the boundary conditions in
a stable manner, we introduce the decomposition
\begin{equation}
  \beta_{x} \Ex + \beta_{y} \Ey = \mat{E}_{+} + \mat{E}_{-},\label{eq:E_decomp}
\end{equation}
where $\mat{E}_{+}$ is symmetric positive semi-definite and $\mat{E}_{-}$ is
symmetric negative semi-definite, and
\begin{equation*}
  \begin{split}
   \pk\Tr\mat{E}_{\pm}\pM
    =\displaystyle\oint_{\Gamma_{\pm}}\fpk\fpm\left( \beta_{x} n_{x} +
    \beta_{y} n_{y} \right) \mr{d}\Gamma,\qquad
     \forall\; k,m \in \{ 1,2,\dots,\nmin{\tauEx} \}.
  \end{split}
\end{equation*}
The existence of the decomposition \eqref{eq:E_decomp} is guaranteed provided
$\mat{F}_{x} = \mat{F}_{y} = \mat{0}$; see Appendix~\ref{sec:E_decomp} for the proof.

Using $\Dx$ and $\Dy$ and the matrix $\mat{E}_{-}$, a consistent
semi-discretization of \eqref{eq:linear_advect} is given by
\begin{equation}
\frac{d \bm{u}}{d t} + \beta_{x} \Dx \bm{u} + \beta_{y} \Dy \bm{u}
= \sigma \M^{-1} \mat{E}_{-} \left(\bm{u} - \bm{u}_{\text{bc}} \right). \label{eq:SBPSAT}
\end{equation}
The three terms on the left-hand side of \eqref{eq:SBPSAT} correspond to the
three terms in the strong-form of the PDE.  The terms on the right-hand side of
\eqref{eq:SBPSAT} are penalties that enforce the boundary conditions weakly
using simultaneous-approximation terms (SATs)~\cite{funaro:1988, Carpenter1994}.
The boundary data is supplied by the vector $\bm{u}_{\text{bc}}$, which must
produce a sufficiently accurate reconstruction of $\fnc{U}_{\text{bc}}$ along
the boundary.  Evaluating $\fnc{U}_{\text{bc}}$ is adequate when the nonzeros of
$\Ex$ correspond to nodes that lie on $\Gamma$, such as the simplex operators
presented below.  More generally, a preprocessing step can be performed to find
a suitable $\bm{u}_{\text{bc}}$.

We now show that \eqref{eq:SBPSAT} is time stable.

\begin{theorem}\label{energy_est}
  Let $\bm{u}$ be the solution to \eqref{eq:SBPSAT} with homogeneous boundary
  conditions and bounded initial condition.  Then the norm $\|\bm{u}\|_{\M} =
  \sqrt{\bm{u}\Tr \M \bm{u}}$ is non-increasing if $\sigma \geq \frac{1}{2}$.
\end{theorem}

\begin{proof}
  Multiplying the semi-discretization (with $\bm{u}_{\text{bc}} = \bm{0}$) from
  the left by $\bm{u}\Tr \M$ we find
  \begin{gather*}
    \bm{u}\Tr\M \frac{d \bm{u}}{dt} 
    + \beta_{x} \bm{u}\Tr \Qx \bm{u} + \beta_{y} \bm{u}\Tr \Qy \bm{u}
    = \sigma \bm{u}\Tr \mat{E}_{-} \bm{u} \\
    \Rightarrow\qquad
    \frac{d \| \bm{u} \|_{\M}^{2}}{dt}
    = - \bm{u}\Tr \mat{E}_{+} \bm{u}
    + (2 \sigma - 1) \bm{u}\Tr \mat{E}_{-} \bm{u}.
  \end{gather*}
  To obtain the last line, we used the definitions of $\Qx$ and $\Qy$, as well
  as the decomposition~\eqref{eq:E_decomp}.  The first quadratic form on the
  right is non-positive by definition of $\mat{E}_{+}$.  The result follows if
  $\sigma \geq \frac{1}{2}$, since $\mat{E}_{-}$ is negative semi-definite.
\qquad\end{proof}

\section{Constructing the operators}\label{sec:construct}

This section describes how we construct \linebreak diagonal-norm SBP operators
for triangles and tetrahedra.  The algorithms described below have been
implemented in the Julia package
SummationByParts\footnote{\url{https://github.com/OptimalDesignLab/SummationByParts.jl}}.

\subsection{The node coordinates and the norm matrix}\label{sec:norm_and_nodes}

Theorem~\ref{L2diagH} tells us that the diagonal entries in $\M$ are positive
weights from a cubature that is exact for polynomials of total degree $2p-1$.
Thus, our first task is to find cubature rules with positive weights for the
triangle and tetrahedron.  Additionally, we seek rules that use as few nodes as
possible for a given order of accuracy while respecting the symmetries of the
triangle and tetrahedron; the former condition is for efficiency while the
latter condition is to reduce directional biases.

For the operators considered in this work, we require that ${ p+d-1 \choose
  d-1}$ cubature nodes lie on each boundary facet, where $d$ is the spatial
dimension.  This requirement on the nodes is motivated by the particular form of
the $\Ex$, $\Ey$, and $\Ez$ operators that we consider below; however,
Definition \ref{DEFSBPgen2} does not require a prescribed number of boundary
nodes, and SBP operators for the 2- and 3-simplex may exist that do not have any
boundary nodes at all.

\afterpage{%
\begin{table}[t!]
  \begin{center}
    \caption[]{Active orbits and their node counts for triangular-element
      operators.  The notation $\textsf{Perm}$ indicates that every permutation
      of the barycentric coordinates is to be considered.  Free-node counts are
      decomposed into the product of the number of nodes in the orbit times the
      number of orbits of that type.
      \label{tab:orbits_tri}}
    \begin{threeparttable}
      \begin{tabular}{lllcccc}
        \rule{0ex}{3ex}& & & \multicolumn{4}{c}{\textbf{operator degree, $p$}} \\\cline{4-7}
        \rule{0ex}{3ex}& \textbf{orbit name} & \textbf{barycentric form} 
        & 1 & 2 & 3 & 4\\\hline
        \rule{0ex}{3ex}\textbf{fixed nodes} & vertices & $\textsf{Perm}(1,0,0)$
        & 3 & 3 & 3 & 3 \\[1ex]
        & mid-edge & $\textsf{Perm}\left(\frac{1}{2},\frac{1}{2},0\right)$
        & --- & 3 & --- & 3 \\[1ex]
        & centroid & $\left(\frac{1}{3},\frac{1}{3},\frac{1}{3}\right)$
        & --- & 1 & --- & --- \\[1.5ex]\hline
        \rule{0ex}{3ex}\textbf{free nodes} & edge & 
        $\textsf{Perm}\left(\alpha, 1-\alpha, 0\right)$
        & --- & --- & $6\times 1$ & $6\times 1$ \\[1ex]
        & $S_{21}$ & $\textsf{Perm}\left(\alpha, \alpha, 1-2\alpha\right)$
        & --- & --- & $3\times 1$ & $3\times 2$ \\[1.5ex]\hline
        \rule{0ex}{3ex}& & \textbf{\# free parameters}
        & --- & --- & 2 & 3 \\
        & & \textbf{\# nodes total}
        & 3 & 7 & 12 & 18
      \end{tabular}
    \end{threeparttable}
  \end{center}
\end{table}
\begin{figure}[h!]
  \renewcommand{\thesubfigure}{}
  \subfigure[$p=1$ \label{fig:tri_nodes_p1}]{%
    \includegraphics[width=0.24\textwidth]{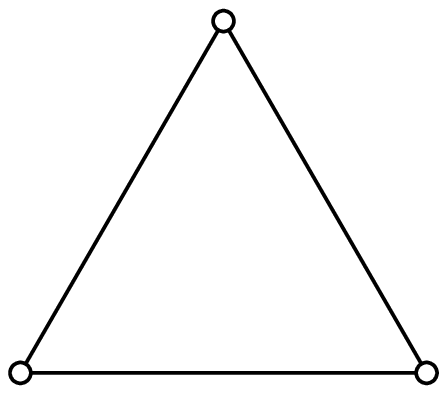}}
  \subfigure[$p=2$ \label{fig:tri_nodes_p2}]{%
        \includegraphics[width=0.24\textwidth]{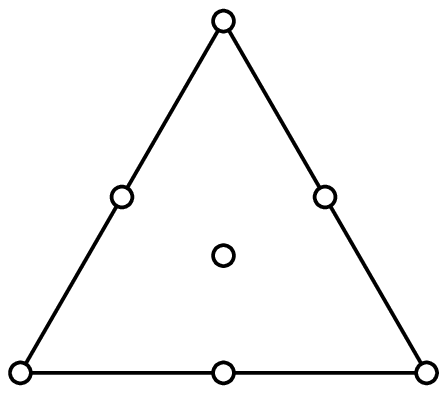}}
  \subfigure[$p=3$ \label{fig:tri_nodes_p3}]{%
        \includegraphics[width=0.24\textwidth]{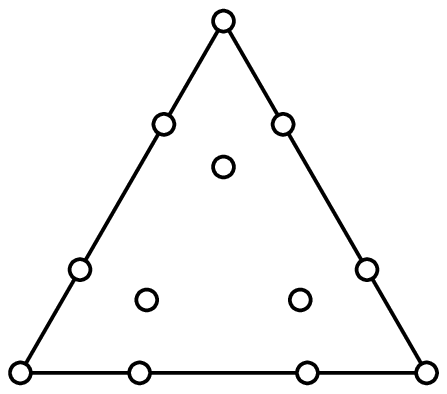}}
  \subfigure[$p=4$ \label{fig:tri_nodes_p4}]{%
        \includegraphics[width=0.24\textwidth]{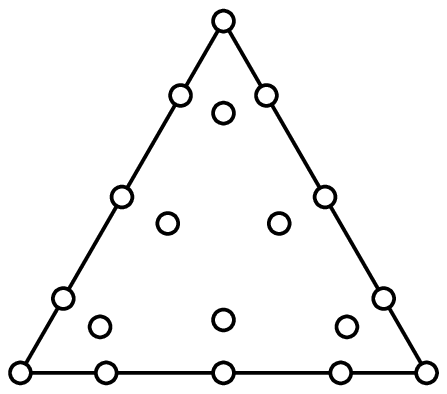}}
  \caption{Node distributions for cubature rules adopted for the SBP operators
    on triangles.\label{fig:tri_nodes}}
\end{figure}
}

Cubature rules that meet our requirements for triangular elements are presented
in references~\cite{liu:1998, Cohen2001, Mulder2001, Giraldo2006} in the context
of the spectral-element and spectral-difference methods.
Table~\ref{tab:orbits_tri} summarizes the rules that are adopted for
triangular-element SBP operators of degree $p=1,2,3$, and 4.  For reference, the
node locations for the triangular cubature rules are shown in
Figure~\ref{fig:tri_nodes}.

To find cubature rules for the tetrahedron, we follow the ideas presented in \cite{Giraldo2006, zhang:2009, witherden:2014}.  Our procedure is briefly outlined below for completeness, but we make no claims regarding the novelty of the cubature rules or our method of finding them.

We assume that each node belongs to a (possibly degenerate) symmetry orbit~\cite{zhang:2009}.  As indicated above, we assume that the cubature node set includes $p+1$ nodes along each edge and $(p+1)(p+2)/2$ nodes on each triangular face.  For the interior nodes, we activate the minimum number of symmetry orbits necessary to satisfy the accuracy conditions; these orbits have been identified through trial and error.

Each symmetry orbit has a cubature weight associated with it, and orbits that are non-degenerate are parameterized using one or more barycentric parameters.  Together, the orbit parameters and the weights are the degrees of freedom that must be determined.  They are found by solving the nonlinear accuracy conditions using the Levenberg-Marquardt algorithm.  The accuracy conditions are implemented using the integrals of orthogonal polynomials on the tetrahedron~\cite{Koornwinder1975, Dubiner1991}.

Table~\ref{tab:orbits_tet} summarizes the node sets used for the tetrahedron
cubature rules, and Figure~\ref{fig:tet_nodes} illustrates the node coordinates.

\afterpage{%
\begin{table}[t!]
  \begin{center}
    \caption[]{Active orbits and their node counts for tetrahedral-element operators.  See the caption of Table~\ref{tab:orbits_tri} for an explanation of the notation.\label{tab:orbits_tet}}
    \begin{threeparttable}
      \begin{tabular}{lllcccc}
        \rule{0ex}{3ex}& & & \multicolumn{4}{c}{\textbf{operator degree, $p$}} \\\cline{4-7}
        \rule{0ex}{3ex}& \textbf{orbit name} & \textbf{barycentric form} 
        & 1 & 2 & 3 & 4\\\hline
        \rule{0ex}{3ex}\textbf{fixed nodes} & vertices & $\textsf{Perm}(1,0,0,0)$
        & 4 & 4 & 4 & 4 \\[1ex]
        & mid-edge & $\textsf{Perm}\left(\frac{1}{2}, \frac{1}{2}, 0, 0 \right)$
        & --- & 6 & --- & 6 \\[1ex]
        & centroid & $\left(\frac{1}{4},\frac{1}{4},\frac{1}{4}, \frac{1}{4} \right)$
        & --- & 1 & --- & 1 \\[1ex]
        & face centroid & $\textsf{Perm}\left(\frac{1}{3}, \frac{1}{3}, \frac{1}{3}, 0\right)$
        & --- & --- & 4 & --- \\[1.5ex]\hline
        \rule{0ex}{3ex}\textbf{free nodes} & edge & 
        $\textsf{Perm}\left(\alpha, 1-\alpha, 0, 0\right)$
        & --- & --- & $12\times 1$ & $12\times 1$ \\[1ex]
        & face $S_{21}$ & $\textsf{Perm}\left(\alpha, \alpha, 1-2\alpha, 0\right)$
        & --- & --- & --- & $12\times 1$ \\[1ex]
        & $S_{31}$ & $\textsf{Perm}\left(\alpha, \alpha, \alpha, 1 - 3\alpha \right)$
        & --- & --- & $4\times 1$ & $4\times 1$ \\[1ex]
        & $S_{22}$ & $\textsf{Perm}\left(\alpha, \alpha, \frac{1}{2}-\alpha, \frac{1}{2} - \alpha \right)$
        & --- & --- & --- & $6\times 1$ \\[1.5ex]\hline
        \rule{0ex}{3ex}& & \textbf{\# free parameters}
        & --- & --- & 2 & 4 \\
        & & \textbf{\# nodes total}
        & 4 & 11 & 24 & 45
      \end{tabular}
    \end{threeparttable}
  \end{center}
\end{table}
\begin{figure}[h!]
  \renewcommand{\thesubfigure}{}
  \subfigure[$p=1$ \label{fig:tet_nodes_p1}]{%
    \includegraphics[width=0.24\textwidth]{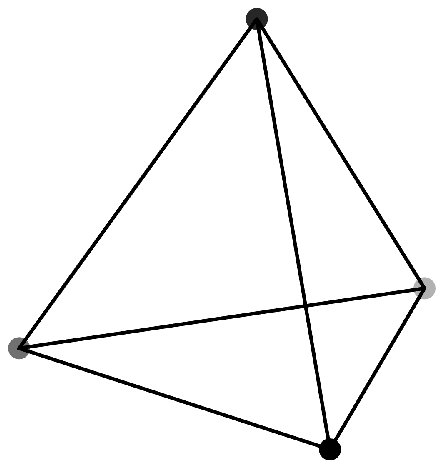}}
  \subfigure[$p=2$ \label{fig:tet_nodes_p2}]{%
        \includegraphics[width=0.24\textwidth]{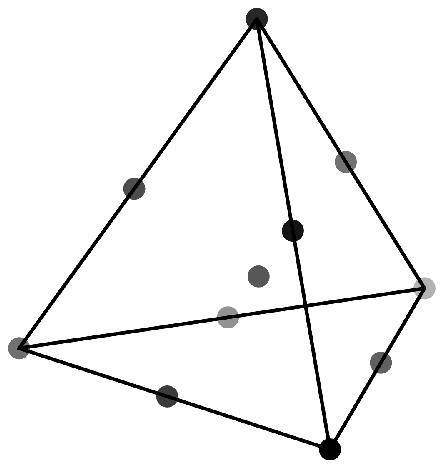}}
  \subfigure[$p=3$ \label{fig:tet_nodes_p3}]{%
        \includegraphics[width=0.24\textwidth]{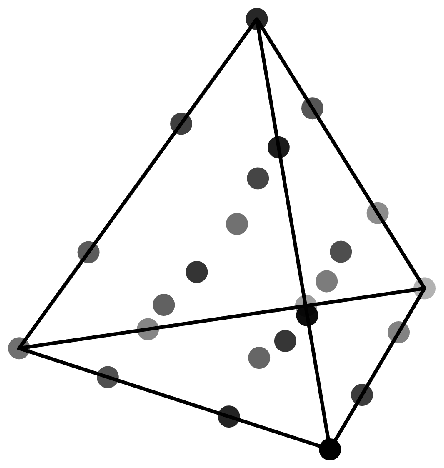}}
  \subfigure[$p=4$ \label{fig:tet_nodes_p4}]{%
        \includegraphics[width=0.24\textwidth]{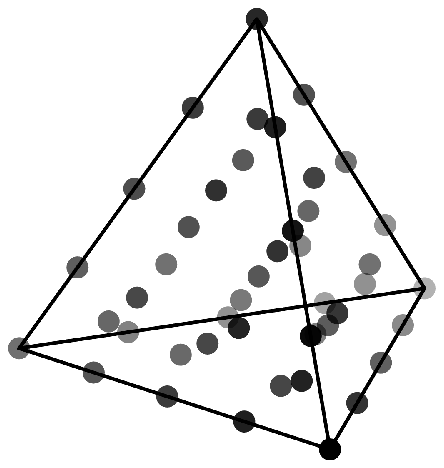}}
  \caption{Node distributions for cubature rules adopted for the SBP operators
    on tetrahedra.\label{fig:tet_nodes}}
\end{figure}
}

\subsection{The boundary operators}

Definition~\ref{DEFSBPgen2} implies that the boundary operator $\Ex$ satisfies
\begin{equation*}
  \bm{v}\Tr \Ex \bm{u} = \oint_{\Gamma} \fnc{U} \fnc{V} n_{x}\, d\Gamma
\end{equation*}
for all polynomials $\fnc{U}$ and $\fnc{V}$ whose total degree is less than
$\tau \geq p$.  In particular, if we choose $\fnc{U}$ and $\fnc{V}$ to be
nodal basis functions on the faces, we can isolate the entries of $\Ex$.  This
is possible because we have insisted on operators with ${ p+d-1 \choose d-1}$
nodes on each facet, which leads to a complete nodal basis.  For further details
on the construction of the boundary operators, we direct the interested reader
to \cite[pg.~187]{hesthaven:2008}.

\begin{remark}
  The boundary operators, when restricted to the boundary nodes, are dense
  matrices.  Contrast this with the tensor-product case, where the boundary
  operators are diagonal matrices.  In the simplex case, we have not found a way
  to construct diagonal $\Ex$, $\Ey$ and $\Ez$ that are sufficiently accurate.
\end{remark}

\subsection{The antisymmetric part}\label{sec:Qx_construct}

The accuracy conditions are used to determine the antisymmetric matrices $\QAx$,
$\QAy$, and $\QAz$.  We will describe the process for $\QAx$, since it can be
adapted in a straightforward way to $\QAy$ and, in the case of the tetrahedron,
$\QAz$.

In theory, we can compute $\QAx$ using the monomials that appear in the SBP
operator definition; however, the monomials are known to produce ill-conditioned
Vandermonde matrices.  Instead, we follow the standard practice in
spectral-element methods and apply the accuracy conditions to appropriate
orthogonal bases on the triangle and tetrahedron~\cite{Koornwinder1975,
  Dubiner1991, hesthaven:2008}.  Unlike finite- and spectral-element methods,
the basis alone does not completely specify an SBP operator.

Let $\P$ and $\Px$ be the matrices whose columns are the orthogonal basis
function values and derivatives, respectively, evaluated at the nodes.  Then the
accuracy conditions imply $\Dx \P = \Px$, or, in terms of the unknown $\QAx$,
\begin{equation*}
  \QAx \P = \M \Px - \frac{1}{2} \Ex \P.
\end{equation*}
This can be recast as the linear system
\begin{equation}
  \mat{A} \bm{q} = \bm{b},
  \label{eq:Qx_equation}
\end{equation}
where $\bm{q}$ denotes a vector whose entries are the strictly lower part of
$\QAx$:
\begin{equation*}
  \bm{q}\left({\scriptstyle \frac{(i-2)(i-1)}{2}} + j\right) = \left(\QAx\right)_{i,j}, \qquad 2 \leq i \leq n, 1 \leq j < i.
\end{equation*}
There are $(n-1)n/2$ unknowns and $n\times {p+d \choose d}$ equations in
\eqref{eq:Qx_equation}; thus, for the operators considered here, there are more
equations than unknowns.  Fortunately, the compatibility conditions ensure that
the system is consistent.  Indeed, for $p \geq 3$ the rank of $\mat{A}$ is
actually less than the size of $\bm{q}$, so there are an infinite number of
solutions.  In these cases, we choose the minimum-norm least-squares
solution~\cite{golub:1996} to \eqref{eq:Qx_equation}.

\ignore{David C. DRF: Why are we using a least square solution? Dont we want to
  satisfy the degree equations to within machine error? In the context of PnPM
  DG methods it has been shown that the least squares solution leads to worse
  operators (I will send the reference later tonight)}

\ignore{JEH: When the problem is rank-deficient, as it is here, the minimum-norm
  least-squares solution does satisfy all of the equations}

\subsection{Similarities and differences with existing operators}

\ignore{The $p=1$ operators are identical to mass-lumped linear
  finite-elements.}

There is a vast literature on high-order discretizations for simplex elements,
so we focus on the two that share the most in common with the proposed SBP
operators: the diagonal mass-matrix spectral-element (SE) method
\cite{Cohen2001, Mulder2001, Giraldo2006} and the spectral-difference (SD)
method~\cite{Liu2006}.

Our norm matrices $\M$ are identical to the lumped mass matrices in the SE
method.  The difference between the methods arises in the definition of $\Qx$.
In the SE method of Giraldo and Taylor~\cite{Giraldo2006}, the $\Qx$ matrix is
defined as
\begin{equation*}
  \left(\Qx^{\text{SE}}\right)_{i,j} = \sum_{k=0}^{n} \M_{k,k} \phi_{i}(x_{k},y_{k}) \frac{\partial \phi_{j}}{\partial x}(x_{k},y_{k}),
\end{equation*}
where $\left\{\phi_{i}\right\}_{i=1}^{n}$ is the so-called cardinal basis.  For
a degree $p$ operator, the cardinal basis is a polynomial nodal basis that is a
super-set of the basis for degree $p$ polynomials; the basis contains polynomials
of degree greater than $p$, because the number of nodes $n$ is greater than
${p+d \choose d}$, in general.  Consequently, the cubature defined by $\M$ is
not exact for the product $\phi_{i} \partial \phi_{j}/\partial x$ when $p \geq 2$,
and the resulting $\Qx^{\text{SE}}$ does not satisfy the SBP definition for the
$p \geq 2$ discretizations.  Indeed, as the results below demonstrate, the
higher-order SE operators are unstable and require filtering or numerical
dissipation \emph{even for linear problems}.

\begin{remark}
  The $\Qx^{\text{SE}}$ matrix in the diagonal mass-matrix SE method can be made
  to satisfy the SBP definition by using a cubature rule that is exact for the
  cardinal basis; however, such a cubature rule would require additional
  cubature points and would defeat the purpose (\ie efficiency) of collocating
  the cubature and basis nodes.
\end{remark}

Diagonal-norm SBP operators can also be viewed as a special case of the SD
method in which the unknowns and fluxes are collocated.  As pointed out in
\cite{Liu2006}, this means that our proposed SBP operators require more unknowns
to achieve a given accuracy than spectral-difference methods; however,
collocation eliminates the reconstruction step, so there is a tradeoff between
memory and floating-point operations.  More importantly, this relative increase
in unknowns applies only to discontinuous methods.  If we assemble a global SBP
operator, as described below, then the number of unknowns can be significantly
reduced.

\subsection{Assembly of global SBP operators from elemental operators}\label{sec:construct_global}

The SBP operators defined in Sections
\ref{sec:norm_and_nodes}--\ref{sec:Qx_construct} can be used in a nodal DG
formulation~\cite{hesthaven:2008} with elements coupled weakly using, for
example, simultaneous approximation terms~\cite{funaro:1988, Carpenter1994}.  An
alternative use for these element-based operators, and the one pursued here, is
to mimic the continuous Galerkin formulation.  That is, we assemble global SBP
operators from the elemental ones.

We need to introduce some additional notation to help describe the assembly
process and facilitate the proof of Theorem~\ref{thrm:global_assembly} below.
Suppose the domain $\Omega$ is partitioned into a set of $L$ non-overlapping
subdomains $\Omega^{(l)}$ with boundaries $\Gamma^{(l)}$:
\begin{equation*}
  \Omega = \bigcup_{l=1}^{L} \bar{\Omega}^{(l)},
  \qquad\text{and}\qquad
  \Omega^{(k)} \cap \Omega^{(l)} = \emptyset, \; \forall\; k \neq l,
\end{equation*}
where $\bar{\Omega}^{(l)} = \Omega^{(l)} \cup \Gamma^{(l)}$ denotes the closure
of $\Omega^{(l)}$.

Each subdomain is associated with a set of nodes $S^{(l)} \equiv
\{(x_{i}^{(l)},y_{j}^{(l)})\}_{i=1}^{n}$, such that $(x_{i}^{(l)},y_{i}^{(l)})
\in \bar{\Omega}^{(l)}$.  In the present context, some of the nodes in $S^{(l)}$
will lie on the boundary $\Gamma^{(l)}$ and be shared by adjacent subdomains.

Let $S \equiv \cup_{l} S^{(l)}$.  Suppose there are $\nglob$ unique nodes in
$S$, and let each node be assigned a unique global index.  Suppose $\iglob$
is the global index corresponding to the $i^{\text{th}}$ local node of element
$l$.  We define $\mat{Z}^{(l)}(i,j)$ to be the $\nglob \times \nglob$
matrix with zeros everywhere except in the $(\iglob,\jglob)$ entry, which
is unity.  If $e_{\iglob}$ denotes the $\iglob$ column of the
$\nglob\times \nglob$ identity, then $\mat{Z}^{(l)}(i,j) = e_{\iglob}
e_{\jglob}\Tr$.

We can now state and prove the following result.

\begin{theorem}\label{thrm:global_assembly}
  Let $\Dx^{(l)} = \left(\M^{(l)}\right)^{-1} \Qx^{(l)}$ be a degree $p$ SBP operator for
  the first derivative $\partial/\partial x$ on the node set $S^{(l)}$.  If
  \begin{align*}
    \M &\equiv \sum_{l=1}^{L} \sum_{i=1}^{n} \M^{(l)}_{i,i} \mat{Z}^{(l)}(i,i) \\
    \Qx &\equiv \sum_{l=1}^{L} \sum_{i=1}^{n} \sum_{j=1}^{n} \left(\Qx^{(l)}\right)_{i,j} \mat{Z}^{(l)}(i,j),
  \end{align*}
  then $\Dx = \M^{-1} \Qx$ is a degree $p$ SBP operator on the global node set $S$.
\end{theorem}
\begin{proof}
  We need to check each of the three properties in Definition~\ref{DEFSBPgen2}.
  
  \begin{remunerate}
  \item The first property is straightforward, albeit tedious, to verify.
    $\M^{-1}$ exists by property \ref{sbp:H}, which is shown to hold independently below,
    so we have
    \begin{align*}
      \Dx\pk
      &= \M^{-1}\Qx\pk \\
      &= \M^{-1} \left[\sum_{l=1}^{L} \sum_{i=1}^{n} \sum_{j=1}^{n} \left(\Qx^{(l)}\right)_{i,j} \mat{Z}^{(l)}(i,j) \right] \pk \\
      &= \M^{-1} \sum_{l=1}^{L} \sum_{i=1}^{n} \frac{\M^{(l)}_{i,i}}{\M^{(l)}_{i,i}} 
        \sum_{j=1}^{n} \left(\Qx^{(l)}\right)_{i,j} e_{\iglob}\fpk(x_{\jglob},y_{\jglob}) \\
      &= \M^{-1} \sum_{l=1}^{L} \sum_{i=1}^{n} \M^{(l)}_{i,i} e_{\iglob} 
        \sum_{j=1}^{n} \left(\Dx^{(l)}\right)_{i,j} \fpk(x_{j},y_{j}) \\
      &= \M^{-1} \sum_{l=1}^{L} \sum_{i=1}^{n} \M^{(l)}_{i,i} e_{\iglob} \dxfpk(x_{i},y_{i})\\
      &= \M^{-1} \left[\sum_{l=1}^{L} \sum_{i=1}^{n} \M^{(l)}_{i,i} \mat{Z}^{(l)}(i,i) \right]
    \dxpk,
      \end{align*}
    But the term in brackets above is the definition of $\M$, so we are left
    with $\Dx\pk = \dxpk$, as desired.

\ignore{    
  \item We form the decomposition $\Qx = \QAx + \frac{1}{2} \Ex$ where
    \begin{align*}
      \QAx &= \sum_{l=1}^{L} \sum_{i=1}^{n} \sum_{j=1}^{n} \left(\QAx\right)^{(l)}(i,j) \mat{Z}^{(l)}_{i,j}, \\
      \Ex &= \sum_{l=1}^{L} \sum_{i=1}^{n} \sum_{j=1}^{n} \Ex^{(l)}(i,j) \mat{Z}^{(l)}_{i,j}.
    \end{align*}
    The matrix $\QAx$ is antisymmetric, because it is the sum of antisymmetric
    matrices.  Similarly, $\Ex$ is symmetric, because it is the sum of symmetric
    matrices.  Hence, property \ref{sbp:sym} is satisfied.
}

  \item $\M$ is clearly diagonal and positive by construction, so property \ref{sbp:H} is
    satisfied.

  \item Finally, we form the decomposition $\Qx = \QAx + \frac{1}{2} \Ex$ where
    \begin{align*}
      \QAx &= \sum_{l=1}^{L} \sum_{i=1}^{n} \sum_{j=1}^{n} \left(\QAx\right)^{(l)}(i,j) \mat{Z}^{(l)}_{i,j}, \\
      \Ex &= \sum_{l=1}^{L} \sum_{i=1}^{n} \sum_{j=1}^{n} \Ex^{(l)}(i,j) \mat{Z}^{(l)}_{i,j}.
    \end{align*}
    The matrix $\QAx$ is antisymmetric, because it is the sum of antisymmetric
    matrices.  Similarly, $\Ex$ is symmetric, because it is the sum of symmetric
    matrices.  In addition, 
    \begin{align*}
      \pk\Tr \Ex 
      \pM
      &= \sum_{l=1}^{L} \sum_{i=1}^{n} \sum_{j=1}^{n} \Ex^{(l)}(i,j)
      \pk\Tr \mat{Z}^{(l)}_{i,j} \pM \\
      &= \sum_{l=1}^{L} \oint_{\Gamma^{(l)}} \fpk\fpm n_{x}\mr{d}\Gamma \\
      &= \oint_{\Gamma}\fpk\fpm n_{x}\mr{d}\Gamma,
    \end{align*}
    where we have used the fact that the boundary fluxes of adjacent elements
    cancel analytically.  Thus, property~\ref{sbp:Ex} is satisfied.
  \end{remunerate}\qquad\end{proof}

\section{Results}\label{sec:results}

The two-dimensional linear advection equation is used to verify and study the
triangular-element SBP operators of Section~\ref{sec:construct}.  In particular,
we consider the problem
\begin{align*}
  &\frac{\partial \fnc{U}}{\partial t} + \frac{\partial \fnc{U}}{\partial x} +
  \frac{\partial \fnc{U}}{\partial y} = 0,\qquad \forall\; (x,y) \in \Omega = [0,1]^2, \\
  &\fnc{U}(x,0,t) = \fnc{U}(x,1,t),\qquad\text{and}\qquad \fnc{U}(0,y,t) = \fnc{U}(1,y,t), \\
  &\fnc{U}(x,y,0) = \begin{cases}
    1 - (4r^2-1)^5 & \text{if}\; r \leq \frac{1}{2} \\
    1, &\text{otherwise},
  \end{cases}
\end{align*}
where $r(x,y) \equiv \sqrt{(x-\frac{1}{2})^2 + (y-\frac{1}{2})^2}$.  The
boundary conditions imply periodicity in both the $x$ and $y$ directions, and
the PDE implies an advection velocity of $(1,1)$.  The initial condition is a
$C^4$ continuous bump function that is periodic on $\Omega$.

A nonuniform mesh for the square domain $\Omega$ is generated, in order to
eliminate possible error cancellations that may arise on uniform grids.  The
vertices of the mesh are given by
\begin{equation*}
  x_{i,j} = \frac{i}{N} + \frac{1}{40}\sin\left(2\pi i/N\right)\sin\left(2\pi j/N\right),\qquad
  y_{i,j} = \frac{j}{N} + \frac{1}{40}\sin\left(2\pi i/N\right)\sin\left(2\pi j/N\right),
\end{equation*}
where $N$ is the number of elements along an edge, and $i,j = 0,1,2,\ldots,N$.
The nominal element size is $h \equiv 1/N$.  A triangular mesh is generated by
dividing each quadrilateral $\{x_{i,j}, x_{i+1,j}, x_{i,j+1}, x_{i+1,j+1}\}$ along
the diagonal from $x_{i+1,j}$ to $x_{i,j+1}$.  Finally, for an SBP element of
degree $p$, the reference-element nodes are mapped (affinely) to each triangle
in the mesh.  Figure~\ref{fig:mesh} illustrates a representative mesh for $p=3$
and $N=12$.

\begin{remark}
  The use of an affine mapping on each element ensures that the mapping Jacobian
  is element-wise constant; consequently, the transformed PDE has constant
  coefficients in the reference space and the finite-difference operator remains
  SBP.  More generally, curvilinear elements would require constructing SBP operators for each curved element.
\end{remark}

We consider both continuous (C-SBP) and discontinuous (D-SBP) discretizations
using the SBP operators.  The global operators for the C-SBP discretization are
constructed using the assembly process described in
Section~\ref{sec:construct_global}.  In addition, the periodic boundaries are
transparent to the global C-SBP operator, that is, nodes that coincide on the
periodic boundary are considered the same.  The D-SBP discretization of the
surface fluxes follows the nodal DG method outlined, for example, in Hesthaven
and Warburton~\cite[Chapter 6]{hesthaven:2008}.  We use SATs with a penalty
parameter of $\sigma = 1$, which is equivalent to the use of upwind numerical
flux functions.

The classical 4th-order Runge-Kutta scheme is used to discretize the time
derivative.  The maximally stable CFL number for each SBP operator was
identified for $N=32$ using Golden-Section optimization, where the CFL number
was defined as $\sqrt{2} \Delta t/ (h \Delta r) = \sqrt{2} N \Delta t / \Delta
r$ for a time-step size of $\Delta t$; $\Delta r$ denotes the minimum distance
between nodes on a right triangle with vertices at $(0,0)$, $(1,0)$ and $(0,1)$.
Each discretization was run for one period, $t \equiv T = 1\;\textsf{unit}$, and
considered stable if the final $L^2$ norm of the solution was less than or equal
to the initial solution norm.  The results of the optimization are listed in
Table~\ref{tab:cfl}.

\begin{table}[tbp]
  \caption[]{Maximally stable CFL numbers for the SBP operators on
    the nonuniform mesh with $N=32$.\label{tab:cfl}}
  \begin{center}
    \begin{tabular}{lllll}
      & \textbf{p=1} & \textbf{p=2} & \textbf{p=3} & \textbf{p=4} \\\hline
      \rule{0ex}{3ex}C-SBP $\text{CFL}_{\max}$ & 1.885 & 2.257 & 1.816 & 1.570 \\
      D-SBP $\text{CFL}_{\max}$ & 0.696 & 1.269 & 1.157 & 1.148 \\
    \end{tabular}
  \end{center}
\end{table}

\subsection{Accuracy and efficiency studies}

Figures~\ref{fig:accuracy_CG} and \ref{fig:accuracy_DG} plot the normalized
$L^{2}$ error for the C-SBP and D-SBP discretizations, respectively, for
SBP-operator degrees one to four and a range of $N$.  Specifically, if
$\bm{u}_{0}$ is the initial solution, $\bm{u}$ is the solution at $t=T$, and
$\M$ is the appropriate SBP norm, then the error is
\begin{equation*}
  L^2 \; \mathsf{Error} = \sqrt{\left(\bm{u} - \bm{u}_{0}\right)\Tr \M \left(\bm{u} - \bm{u}_{0}\right)}.
\end{equation*}
The error is normalized by the integral norm of the initial condition.  The mesh
resolution ranges from $N=4$ to $N=64$ in increments of 4.  Each case was time
marched using a CFL number of $0.9\text{CFL}_{\max}$, which was sufficiently
small to produce negligible temporal discretization errors.

The results in Figure~\ref{fig:accuracy_DG} indicate that the D-SBP
discretizations and the odd-order C-SBP discretizations exhibit asymptotic
convergence rates of $\text{O}(h^{p+1})$. In contrast, the C-SBP discretizations
based on the $p=2$ and $p=4$ operators have convergence rates of
$\text{O}(h^{p})$. These discretizations experience even-odd decoupling, \ie
checkerboarding, which is illustrated in Figure~\ref{fig:compare} by comparing
the spatial error at $t=T$ for the $p=1$ and $p=2$ C-SBP discretizations; the
$p=1$ error is smooth whereas the $p=2$ error is oscillatory. Checkerboarding is
not unique to the SBP simplex operators, and can be observed with
one-dimensional SBP operators. We are currently investigating methods to address
this issue with the even-order C-SBP discretizations.

\ignore{The results in Figure~\ref{fig:accuracy} demonstrate relative improvement when
the operator degree is increased; however, the convergence rates of the $p=2$ and $p=4$
operators appear to be lower than the expected asymptotic rates, and
the convergence rate of the $p=3$ operator is higher than expected.  Giraldo and
Taylor also observed convergence rates that were inconsistent with the expected
rates when the diagonal mass matrix triangular spectral-element method was
applied to linear advection on the sphere.  The root cause of these unexpected
convergence rates is unclear and will be the focus of future work.
}

\begin{figure}
  \begin{minipage}[t]{0.46\textwidth}
    \centering \includegraphics[width=\textwidth]{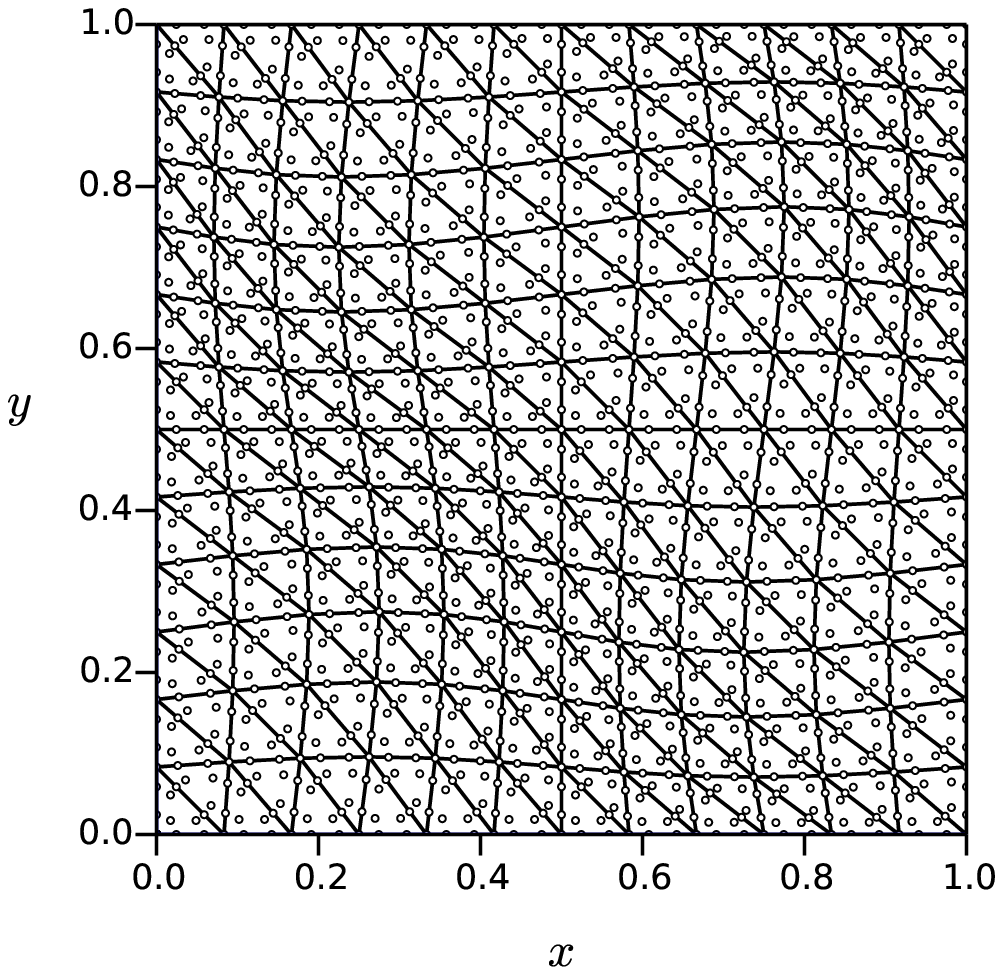}
    \caption{Example mesh with $p=3$ and $N=12$ for accuracy and energy-norm
      studies.\label{fig:mesh}}
  \end{minipage}\hfill
  \begin{minipage}[t]{0.46\textwidth}
    \centering
    \includegraphics[width=\textwidth]{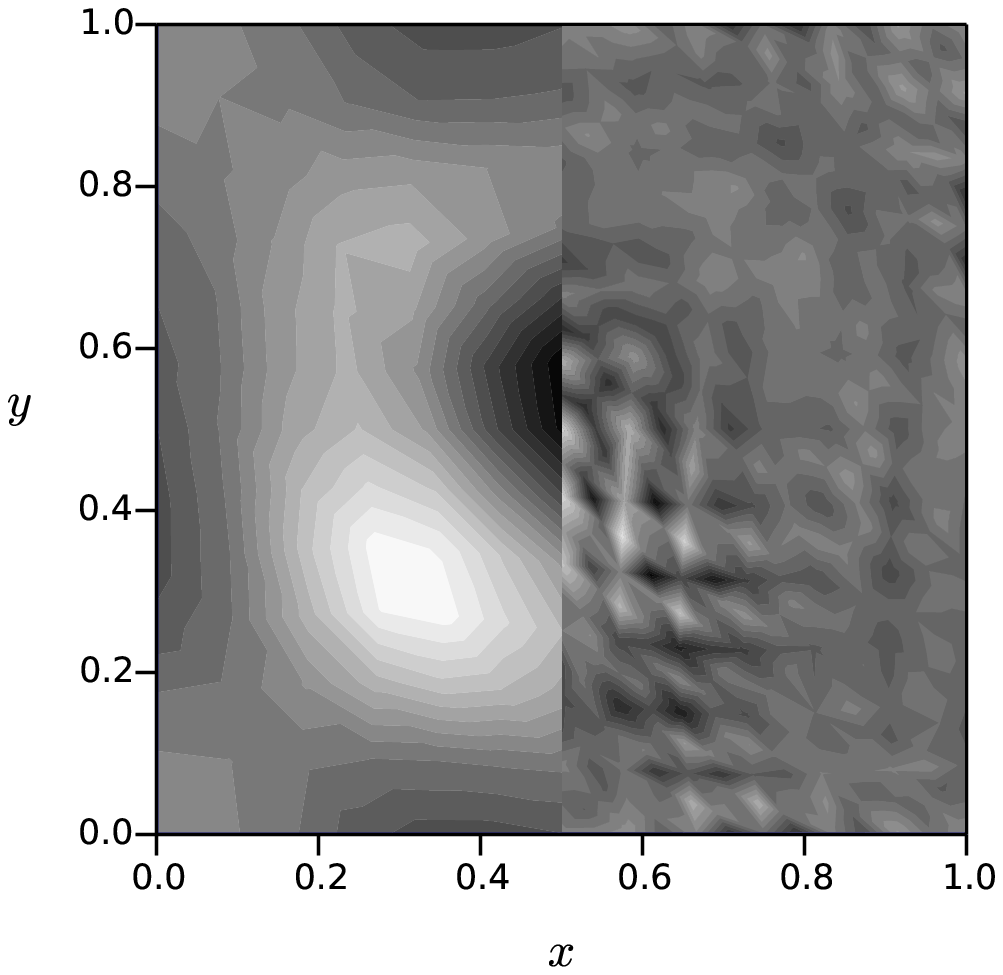}
    \caption{Solution error after one period, $t=T$, for the $p=1$ (left) and
      $p=2$ (right) C-SBP discretizations with $N=12$.\label{fig:compare}}
  \end{minipage}
\end{figure}

\begin{figure}
  \begin{minipage}[t]{0.46\textwidth}
    \centering
    \includegraphics[width=\textwidth]{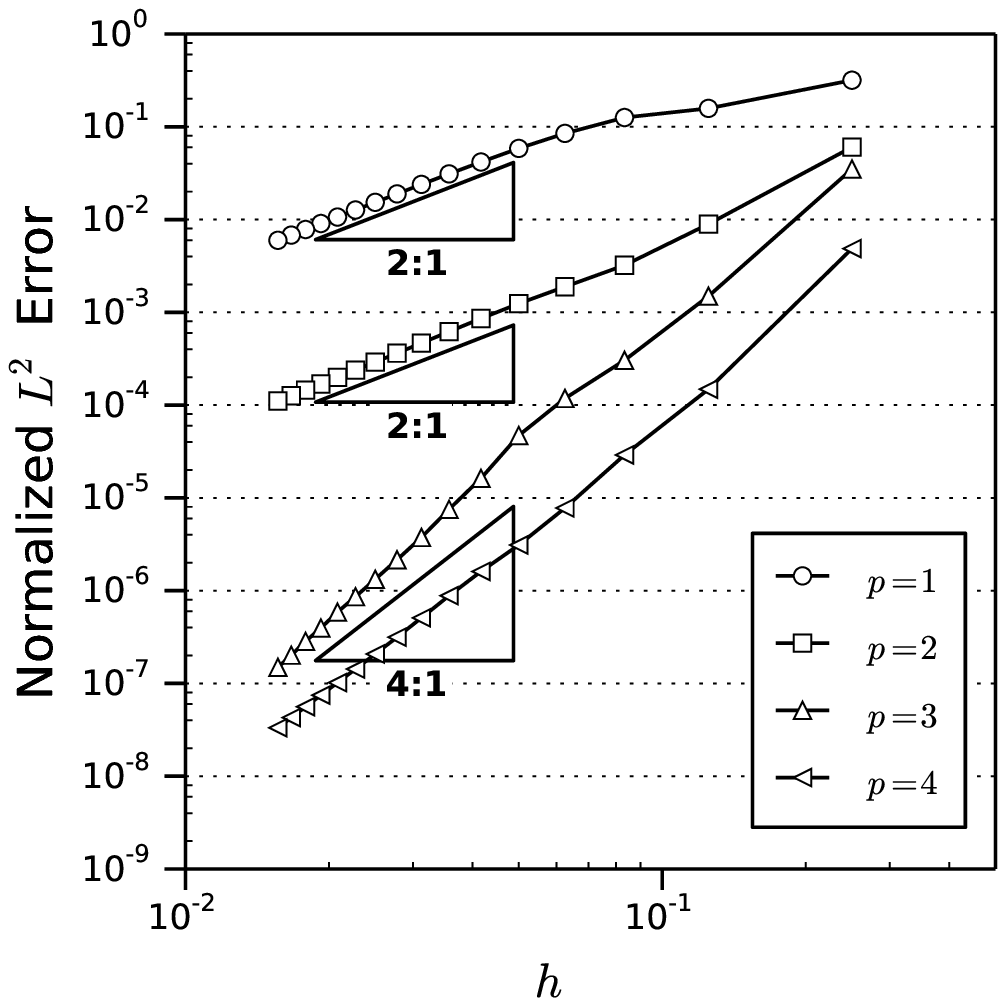}
    \caption{$L^2$ error between the C-SBP solution after one period, $t=T$, and
      the initial condition for different mesh spacing and
      operators.\label{fig:accuracy_CG}}
  \end{minipage}\hfill
  \begin{minipage}[t]{0.46\textwidth}
    \centering
    \includegraphics[width=\textwidth]{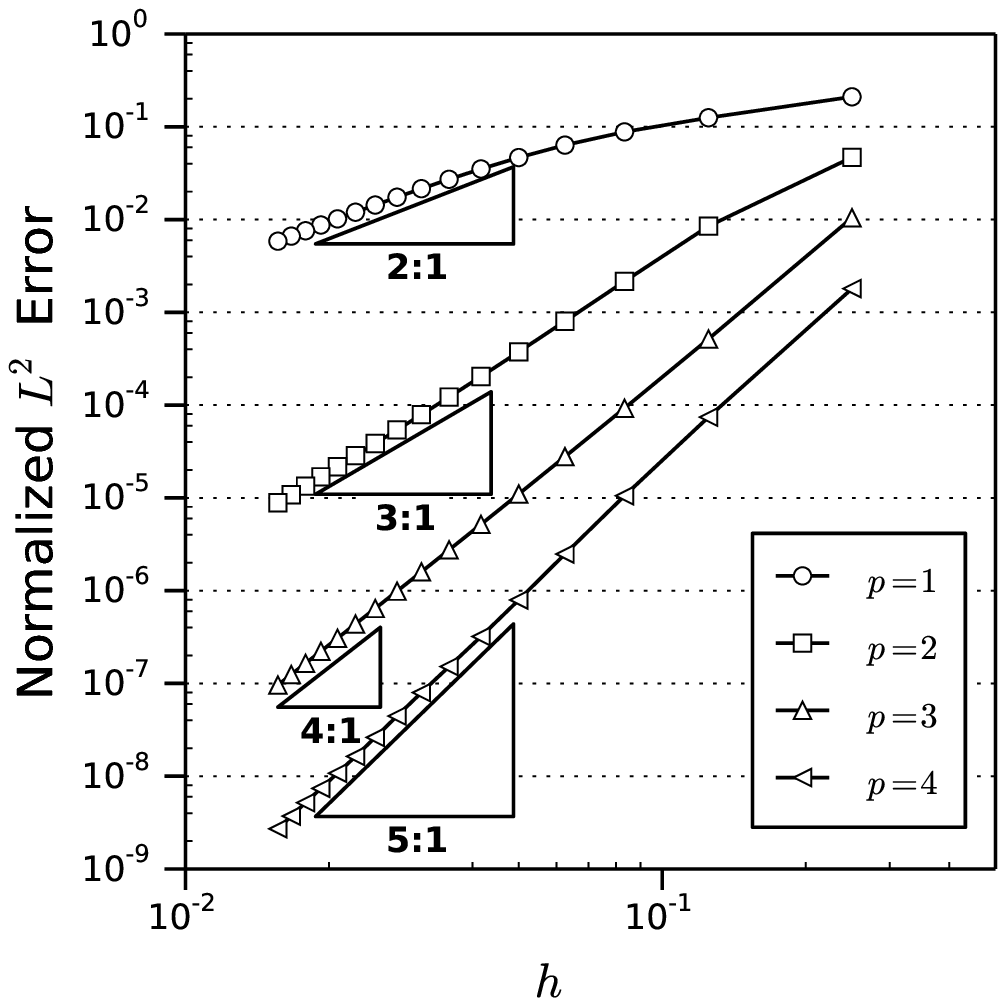}
    \caption{$L^2$ error between the D-SBP solution after one perior, $t=T$, and
      the initial condition for different mesh spacing and
      operators.\label{fig:accuracy_DG}}
  \end{minipage}
\end{figure}

Figures~\ref{fig:efficiency_CG} and \ref{fig:efficiency_DG} plot the normalized
$L^2$ error versus CPU time for both the C-SBP and D-SBP discretizations.  The
runs were performed on an Intel\textsuperscript{\textregistered} Core i5-3570K
processor, and the code was implemented in Julia version 0.4.0.  \ignore{For
  this problem, and despite checkerboarding, the continuous $p=2$ SBP
  discretization is the most efficient over the range of accuracies typical of
  engineering applications.  When the desired relative accuracy is below 0.3\%,
  the $p=4$ discretization is the most efficient.} The efficiency provided by
the high-order operators is apparent in both the continuous and discontinuous
cases.

\begin{figure}
  \begin{minipage}[t]{0.46\textwidth}
    \centering
    \includegraphics[width=\textwidth]{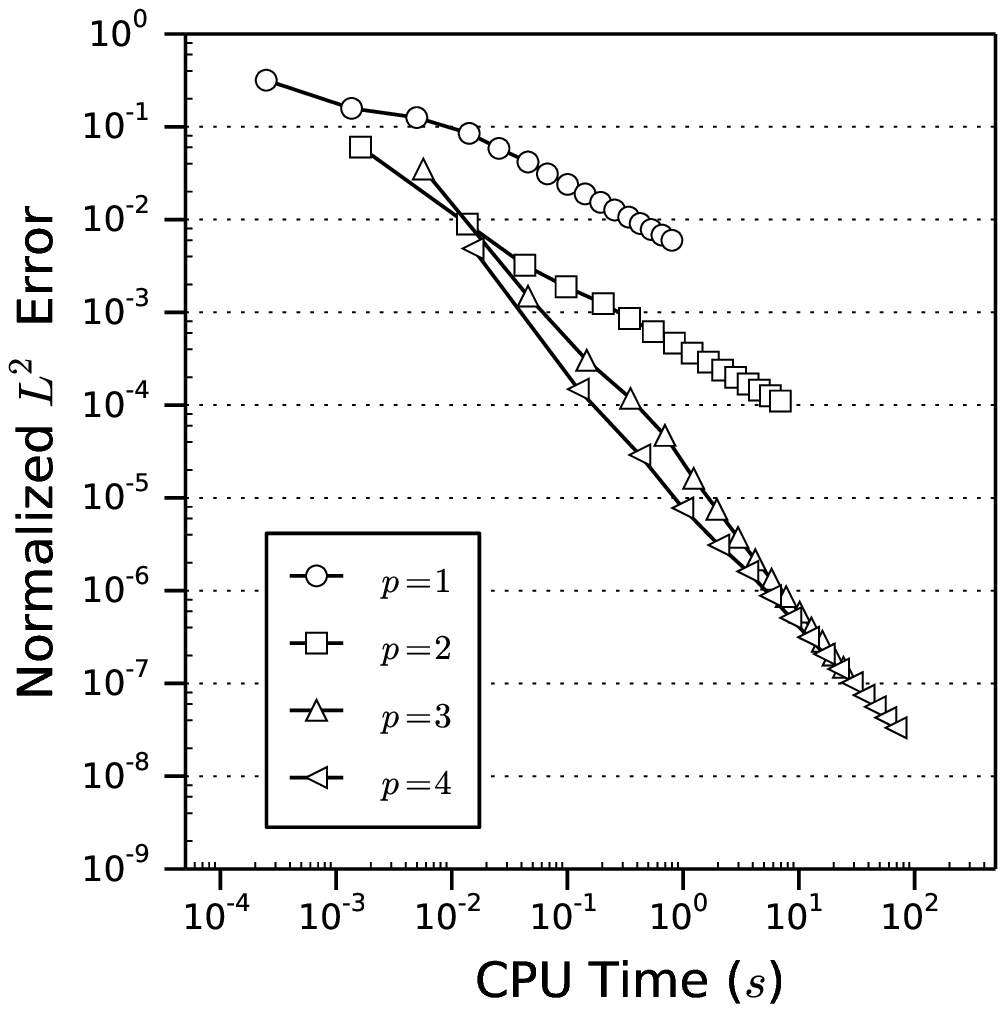}
    \caption{Normalized $L^2$ error of the C-SBP solutions after one period
      versus CPU time measured in seconds.\label{fig:efficiency_CG}}
  \end{minipage}\hfill
  \begin{minipage}[t]{0.46\textwidth}
    \centering
    \includegraphics[width=\textwidth]{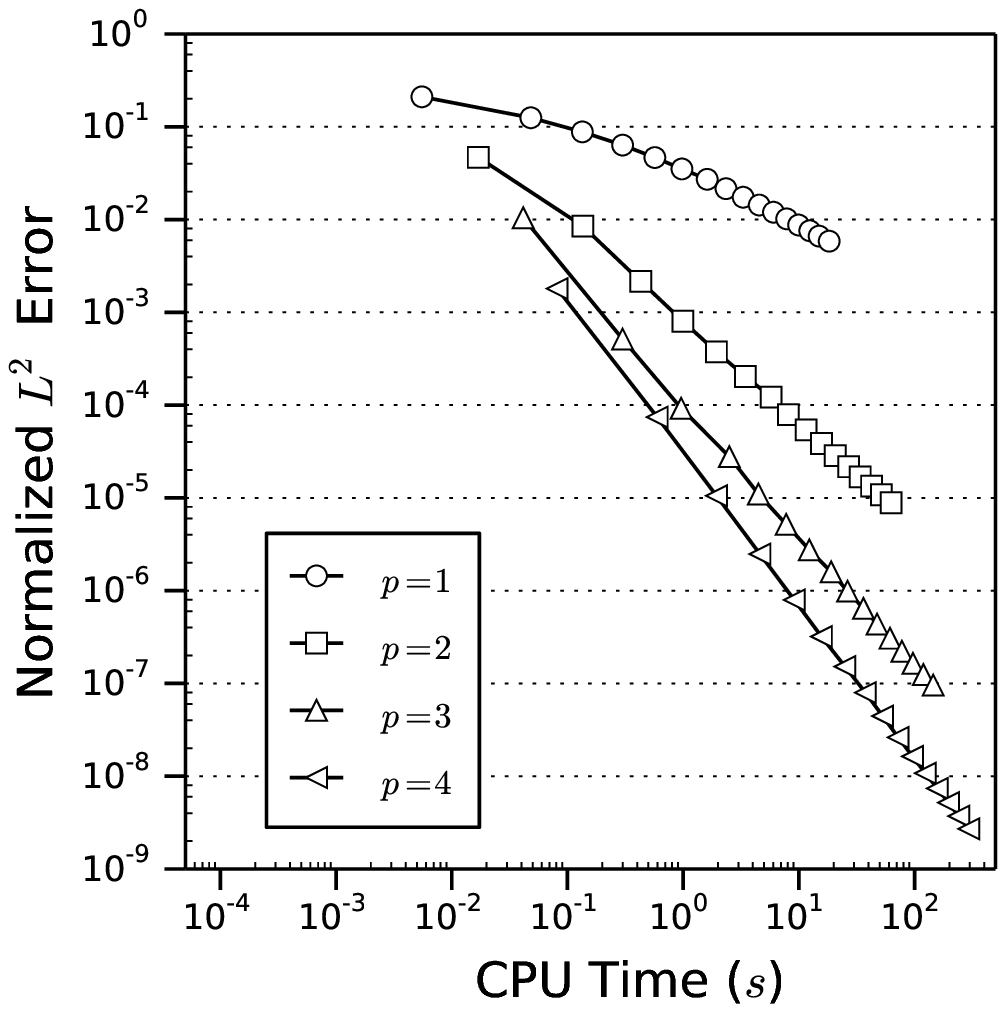}
    \caption{Normalized $L^2$ error of the D-SBP solutions after one period
      versus CPU time measured in seconds.\label{fig:efficiency_DG}}
  \end{minipage}
\end{figure}


\subsection{Stability studies}

Figure~\ref{fig:eigs} shows the spectra of the C-SBP and SE spatial
discretizations for the linear advection problem.  Specifically, these
eigenvalues are for the global operator $\Qx + \Qy$ when $N=12$.  The
eigenvalues of the C-SBP operators are imaginary to machine precision, which
mimics the continuous spectrum for this periodic problem.  This is as expected,
because the boundary operators $\Ex^{(k)}$ cancel between adjacent elements when
the SBP derivative operator is assembled, leaving only the antisymmetric parts.
The SE operator for $p=1$ also has a purely imaginary spectrum, because it is
identical to the linear SBP operator; however, the spectra of the high-order SE
operators have a real component.

The consequences of the eigenvalue distributions are evident when the linear
advection problem is integrated for two periods.  Figure~\ref{fig:energy_hist}
plots the difference between the solution $L^2$ norm at time $t \in [0,2]$ and
the initial solution norm, \ie the change in ``energy'',
\begin{equation*}
  \Delta E = \bm{u}_{n}\Tr \M \bm{u}_{n} - \bm{u}_{0}\Tr \M \bm{u}_{0},
\end{equation*}
where $\bm{u}_{n}$ denotes the discrete solution at time step $n$.  For this
study, $N=12$ and the CFL number was fixed at 0.01 to reduce temporal errors.

The energy history in Figure~\ref{fig:energy_hist} clearly shows that the SE
operators are unstable for this linear advection problem, while the SBP
operators are stable.  The small (linear) decrease in the SBP energy error is due
to temporal errors and can be eliminated by using a different time-marching
method, \eg leapfrog, or at the cost of using a sufficiently small CFL number.

\begin{figure}[tbp]
  \renewcommand{\thesubfigure}{}
  \subfigure[$p=1$ (SBP) \label{fig:eig_p1}]{%
    \includegraphics[width=0.24\textwidth]{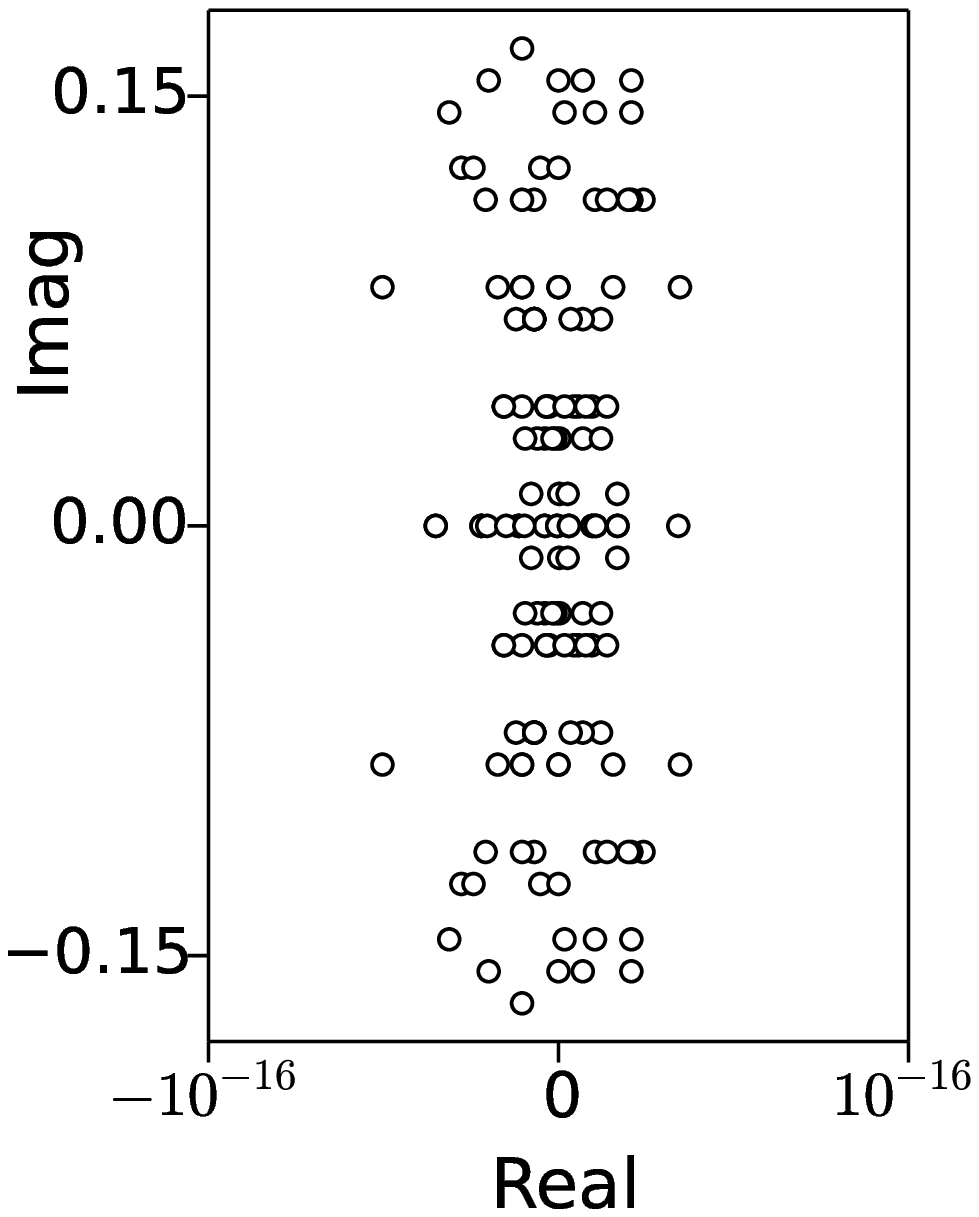}}
  \subfigure[$p=2$ (SBP) \label{fig:eig_p2}]{%
        \includegraphics[width=0.24\textwidth]{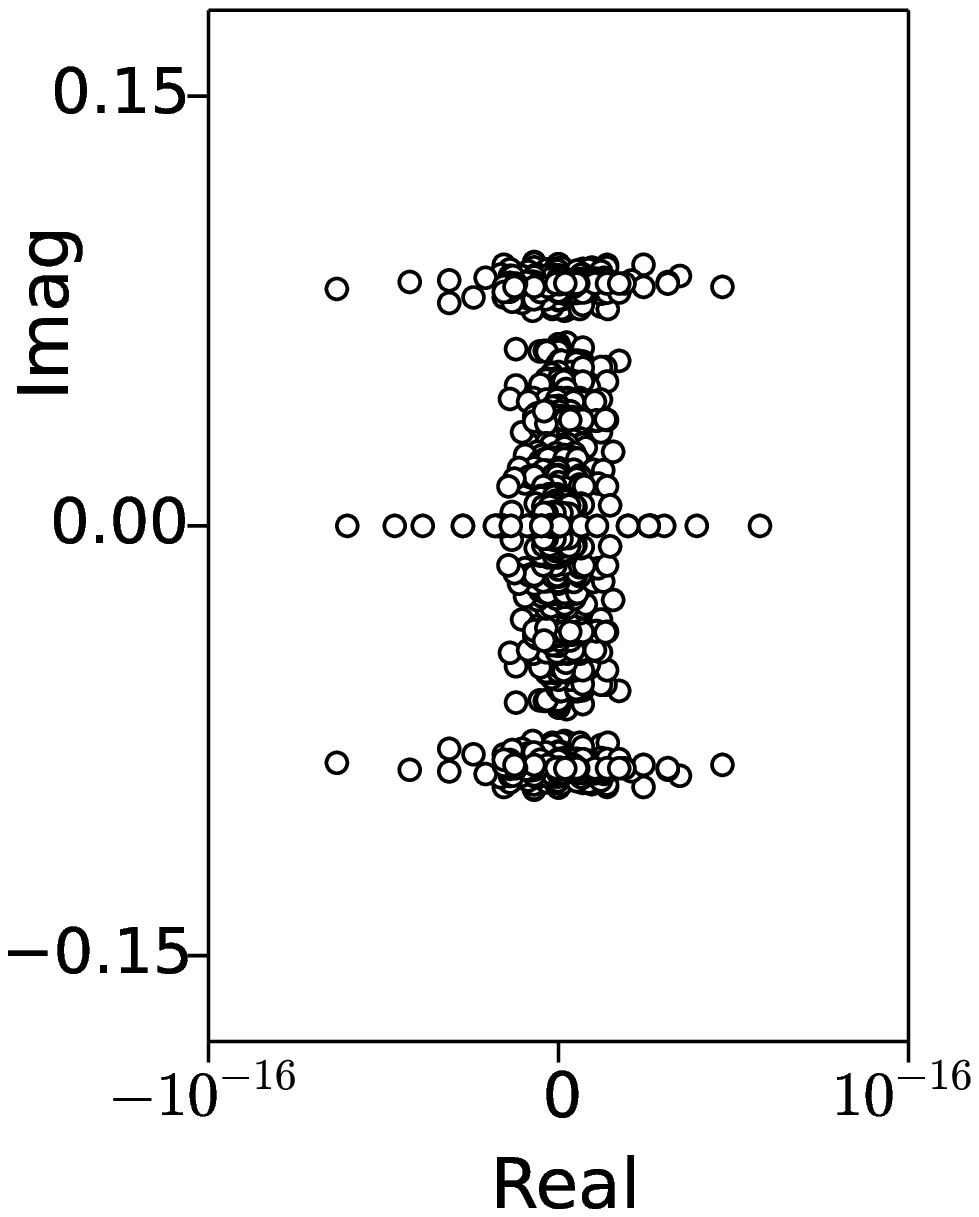}}
  \subfigure[$p=3$ (SBP) \label{fig:eig_p3}]{%
        \includegraphics[width=0.24\textwidth]{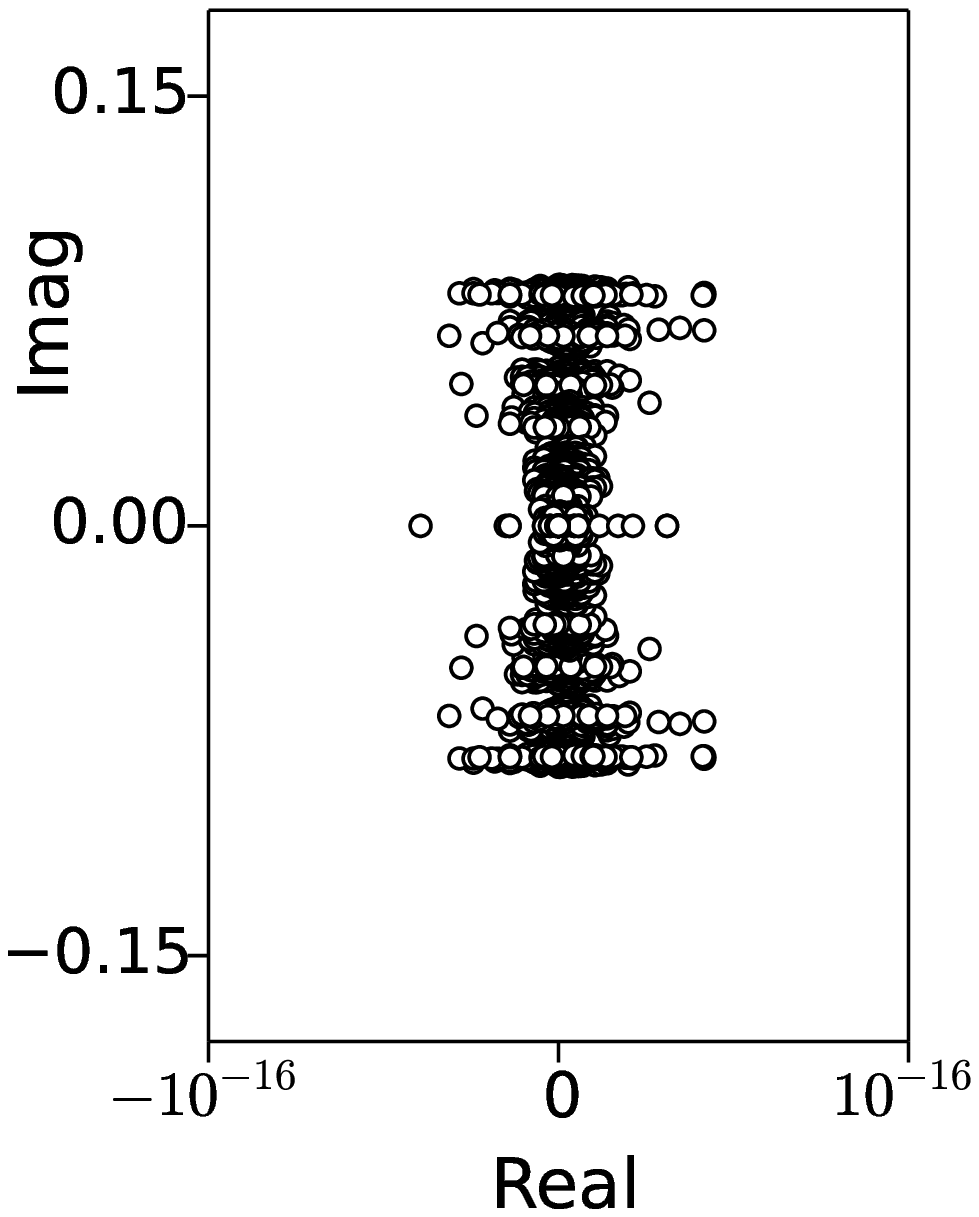}}
  \subfigure[$p=4$ (SBP) \label{fig:eig_p4}]{%
        \includegraphics[width=0.24\textwidth]{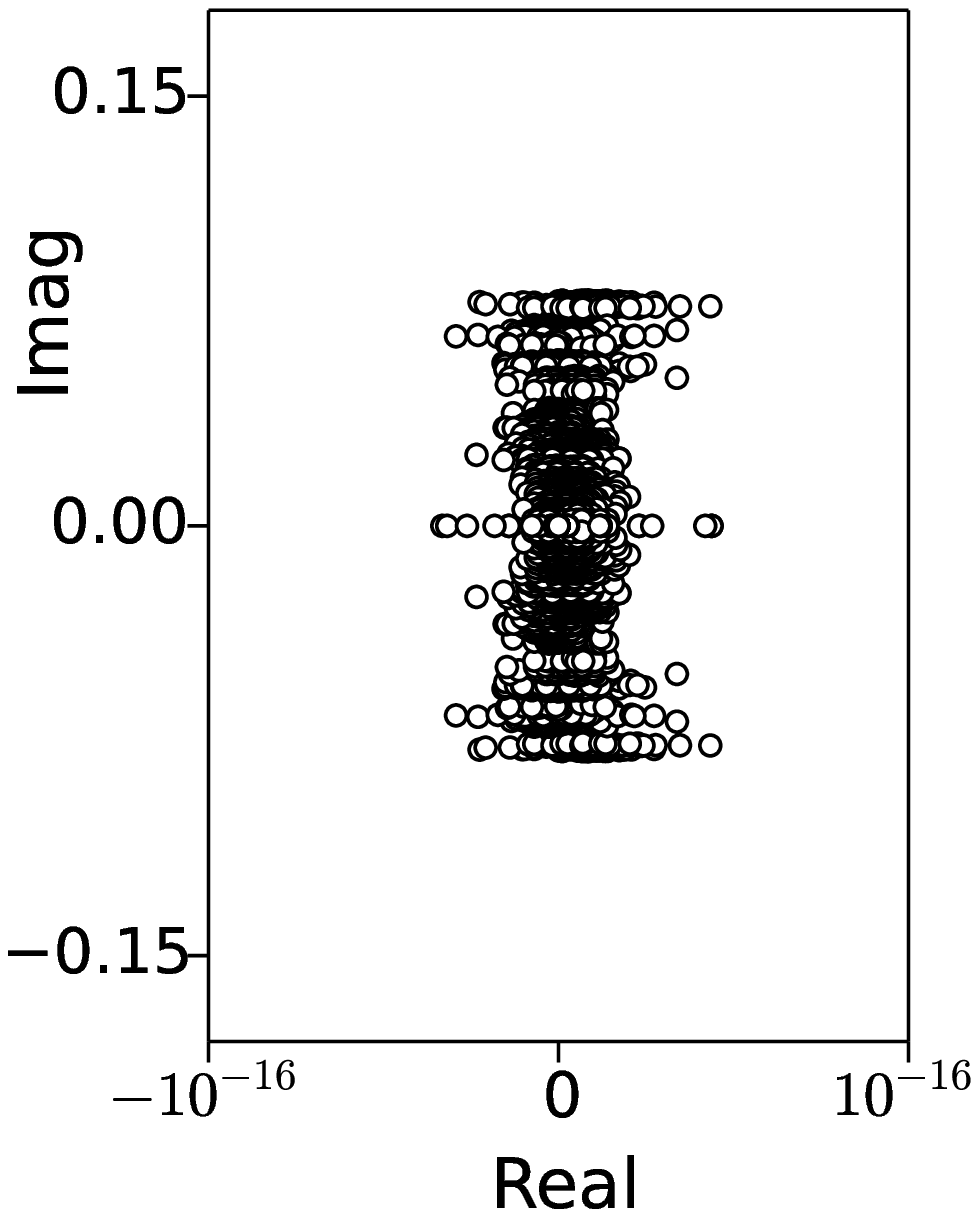}}
  \subfigure[$p=1$ (SE) \label{fig:eig_p1_SE}]{%
    \includegraphics[width=0.24\textwidth]{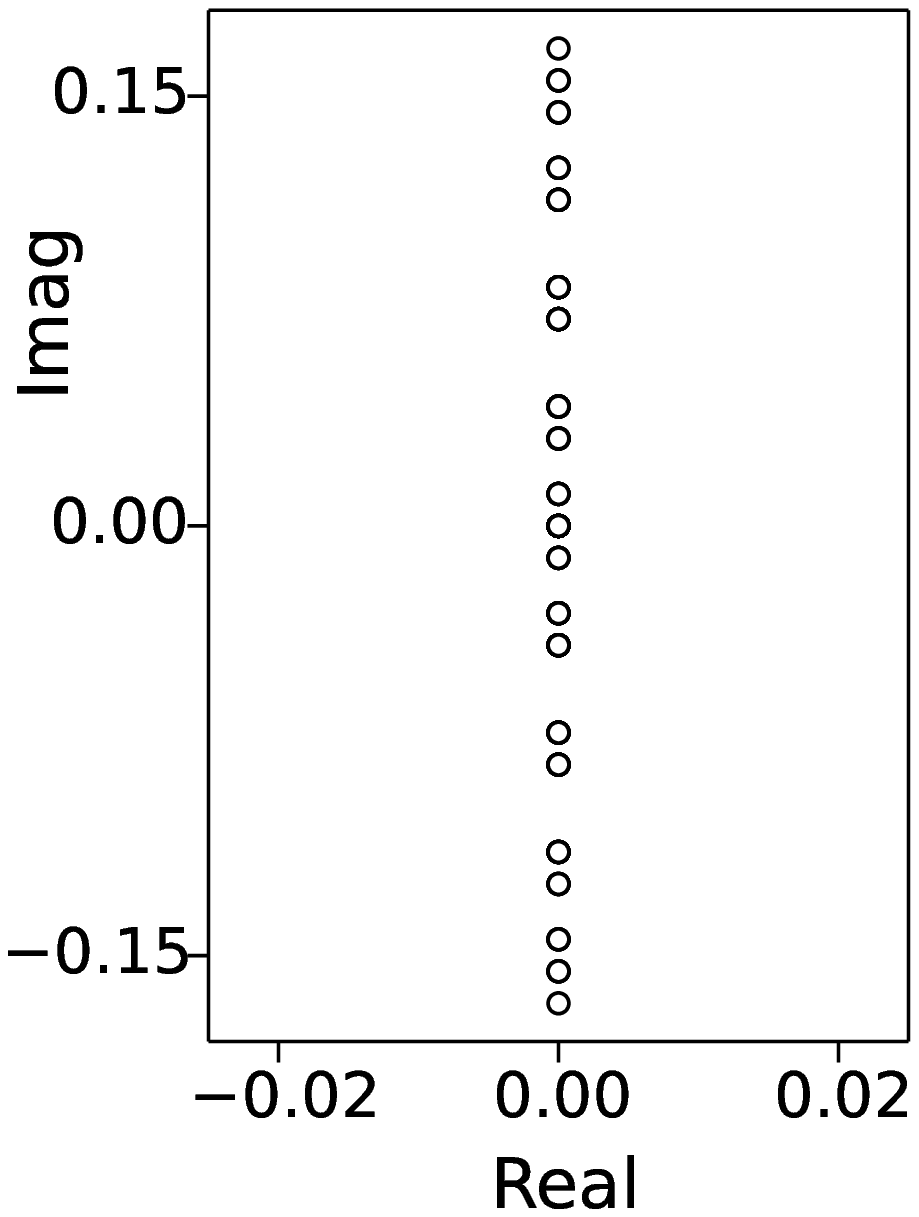}}
  \subfigure[$p=2$ (SE) \label{fig:eig_p2_SE}]{%
        \includegraphics[width=0.24\textwidth]{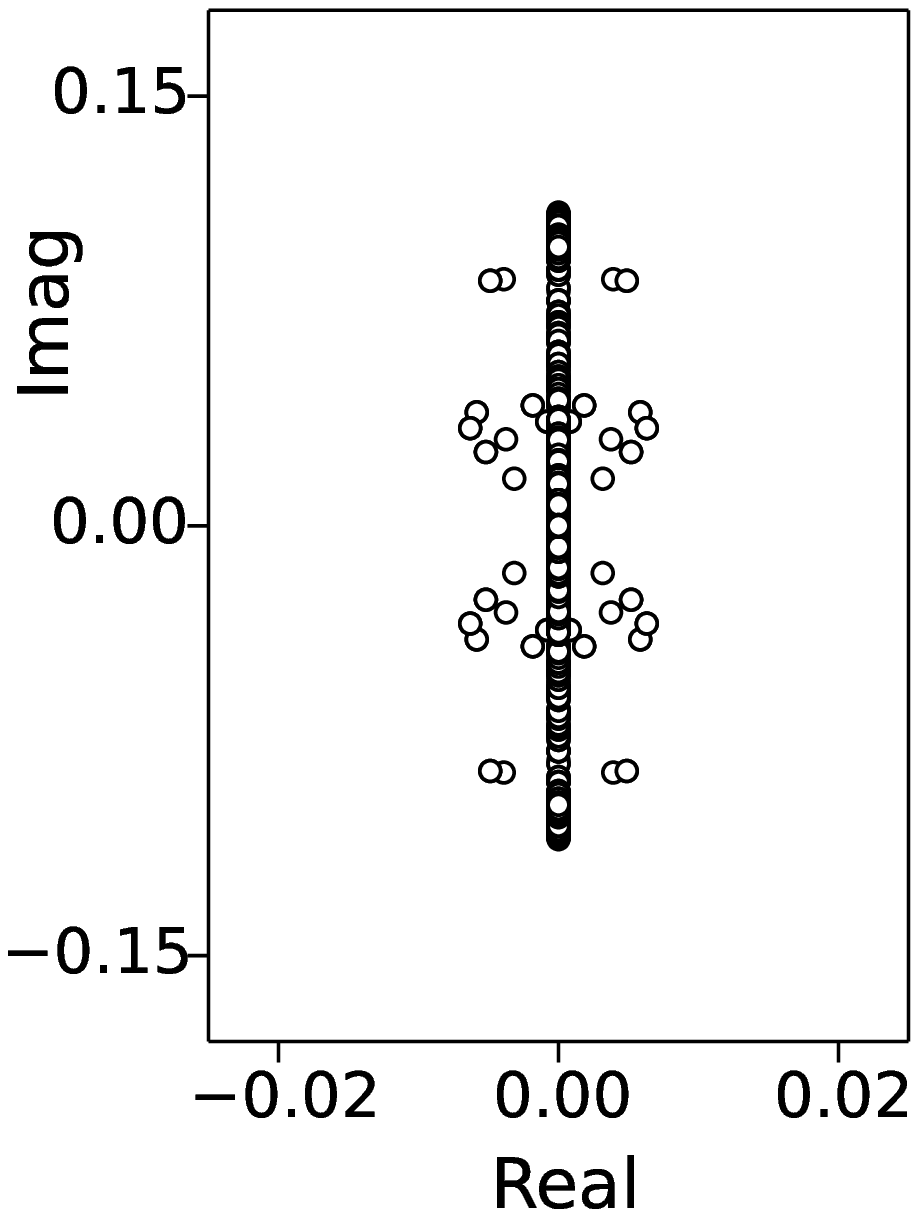}}
  \subfigure[$p=3$ (SE) \label{fig:eig_p3_SE}]{%
        \includegraphics[width=0.24\textwidth]{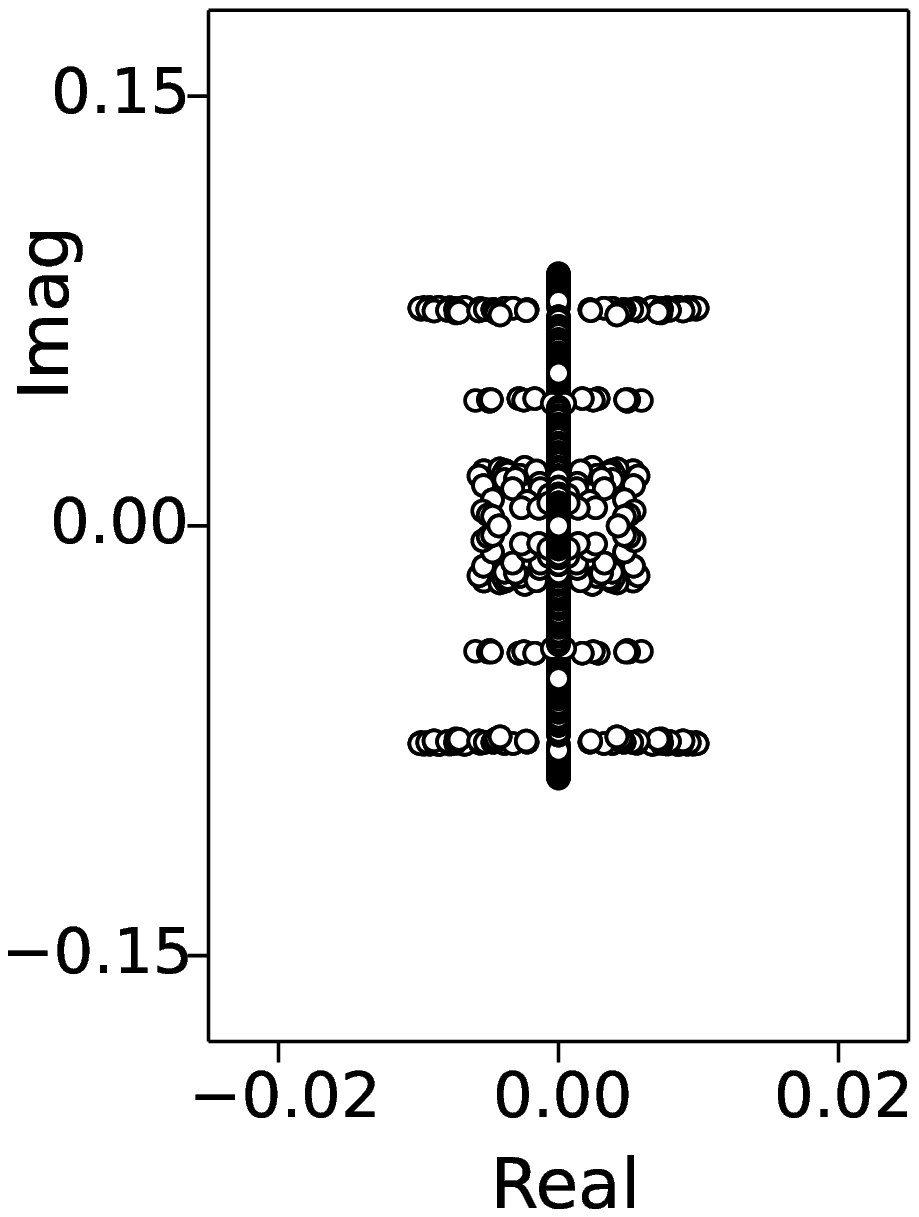}}
  \subfigure[$p=4$ (SE) \label{fig:eig_p4_SE}]{%
        \includegraphics[width=0.24\textwidth]{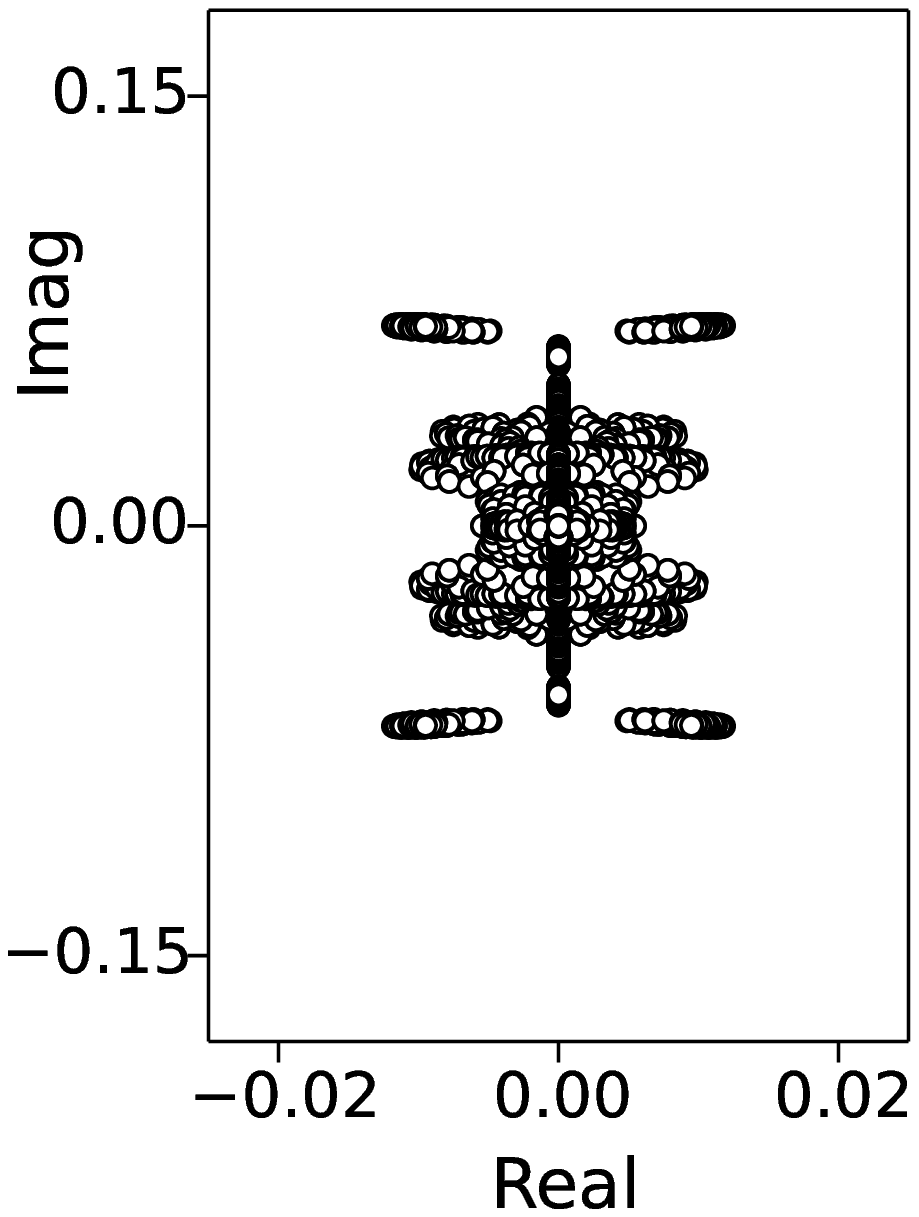}}
  \caption{Eigenvalue distributions for the SBP (upper row) and the SE (lower row) spatial discretizations of the linear advection problem.  Note the different ranges for the real axes.\label{fig:eigs}}
\end{figure}

\begin{figure*}[tbp]
 \begin{center}
 \includegraphics[width=\textwidth]{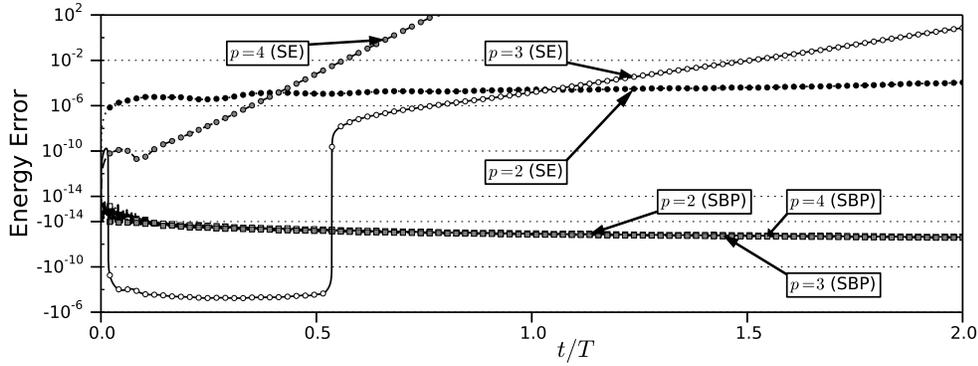}
 \caption[]{Time history of the change in the solution energy norm. Note the use
   of a symmetric logarithmic scale on the vertical axis. \label{fig:energy_hist}}
 \end{center}
\end{figure*}  

\section{Conclusions}\label{sec:conclude}

We proposed a definition for multi-dimensional SBP operators that is a natural
extension of one-dimensional SBP definitions.  We studied the theoretical
implications of the definition in the case of diagonal-norm operators, and
showed that the multi-dimensional operators retain the attractive properties of
tensor-product SBP operators.  A significant theoretical result of this work is
that, for a given domain, a cubature rule with positive weights and a full-rank
Vandermonde matrix (evaluated at the cubature nodes) are necessary and
sufficient for the existence of diagonal-norm SBP operators on that domain.  We
also developed simultaneous approximation terms (SATs) for the weak imposition
of boundary conditions, and we showed that an SBP-SAT discretization of the
linear advection equation is time stable.

We constructed diagonal-norm SBP operators for the triangle and tetrahedron.  To
the best of our knowledge, this is the first example of SBP operators of degree
$p \geq 2$ on these domains.  We also presented an assembly procedure that
constructs SBP operators for a global domain from element-wise SBP operators.

Finally, we verified the triangle-element SBP operators using both continuous
and discontinuous (\ie SAT) inter-element coupling.  Results for linear
advection on a doubly periodic domain demonstrated the time stability and
accuracy of the SBP discretizations.  The results suggest that the proposed
operators could be effective for the long-time simulation of turbulent flows on
complex domains.

\section*{Acknowledgments}

\ignore{J.~E.~Hicken acknowledges the financial support of Rensselaer Polytechnic
Institute, and D.~W.~Zingg acknowledges the support of the Natural Sciences and
Engineering Research Council (NSERC) of Canada.}

All figures were produced using Matplotlib~\cite{Hunter2007}.

\appendix

\section{Decomposition of the SAT matrix $\beta_{x} \Ex + \beta_{y} \Ey$}
\label{sec:E_decomp}
\ignore{ Let $\mat{E} \equiv \beta_{x} \Ex + \beta_{y} \Ey$.  We need to show that this
matrix can be decomposed as $\mat{E} = \mat{E}_{+} + \mat{E}_{-}$ where
$\mat{E}_{+}$ and $\mat{E}_{-}$ are positive semi-definite and negative
semi-definite, respectively, and satisfy the accuracy conditions
\begin{equation}
  \left(\bm{x}^{\ax}\circ\bm{y}^{\ay}\right)\Tr\mat{E}_{\pm}\bm{x}^{\bx}\circ\bm{y}^{\by}
  =\displaystyle\oint_{\Gamma_{\pm}}x^{\ax+\bx}y^{\ay+\by}\left( \beta_{x} n_{x} +
  \beta_{y} n_{y} \right) \mr{d}\Gamma,\qquad 
  \forall\;\ax+\ay,\bx+\by\leq \tauEx \quad k,l=0,1,\dots,\tilde{n},
  \label{eq:E_accuracy}
\end{equation}
}

Let $\mat{E} \equiv \beta_{x} \Ex + \beta_{y} \Ey$.  Recall the block-matrix
definition of $\Ex$ used in the proof of Theorem~\ref{ExistDiagH}.  Using that
definition, and a similar definition for $\Ey$, we have
\begin{equation*}
  \tV\invTr \mat{E} \tV^{-1}
  =
  \tV\invTr 
  \begin{bmatrix}
    \tilde{\mat{E}} & \mat{F}^{T} \\
    \mat{F} & \mat{G} 
  \end{bmatrix} \tV^{-1}
  \equiv
  \beta_{x} \tV\invTr \begin{bmatrix}
    \tEx & \Fx\Tr \\
    \Fx & \Gx 
  \end{bmatrix} \tV^{-1}
  + \beta_{y} \tV\invTr \begin{bmatrix}
    \tEy & \Fy\Tr \\
    \Fy & \Gy
  \end{bmatrix} \tV^{-1}.
\end{equation*}
From the definition of $\tEx$ and $\tEy$, we have that the entries in the block
$\tilde{\mat{E}}$ are given by
\begin{equation*}
  \left(\tilde{\mat{E}}\right)_{k,m}
  = \oint_{\Gamma} \fnc{P}_{k}(x,y) \fnc{P}_{m}(x,y) 
  \left(\beta_{x} n_{x} + \beta_{y} n_{y} \right) \mr{d}\Gamma\qquad 
  \forall\; k,m \in \{ 1,\dots,\nmin{\tau} \},
\end{equation*}
where $\fnc{P}_{k}$ and $\fnc{P}_{l}$ have total degrees less than or equal to
$\tauEx$.  We can decompose $\tilde{\mat{E}}$ by breaking the
above integral into two integrals, one over $\Gamma_{+}$ and one over
$\Gamma_{-}$:
\begin{multline*}
  \left(\tilde{\mat{E}}\right)_{k,m}
  = \oint_{\Gamma_{+}} \fpk\fpm 
  \left( \beta_{x} n_{x} + \beta_{y} n_{y} \right) \mr{d}\Gamma 
  + \oint_{\Gamma_{-}} \fpk\fpm 
  \left(\beta_{x} n_{x} + \beta_{y} n_{y} \right) \mr{d}\Gamma
  = \tilde{\mat{E}}_{+} + \tilde{\mat{E}}_{-},\\
   \forall\; k,m \in \{ 1,\dots,\nmin{\tau}\}.
\end{multline*}
where $\tilde{\mat{E}}_{+}$ and $\tilde{\mat{E}}_{-}$ are equated with the
integrals over $\Gamma_{+}$ and $\Gamma_{-}$, respectively.

\begin{lemma}\label{lem:tilde_E_def}
The matrix $\tilde{\mat{E}}_{-}$ is negative semi-definite and the matrix
$\tilde{\mat{E}}_{+}$ is positive semi-definite.
\end{lemma}

\begin{proof}
  We prove the result for $\tilde{\mat{E}}_{-}$, since the proof for the
  positive-definite matrix is analogous.  Let $\bm{u} \in \mathbb{R}^{\nmin{p}}$
  be an arbitrary nonzero vector, and let $u_{k}$ denote its entries.  Then
  \begin{align*}
    \bm{u}^{T} \tilde{\mat{E}}_{-} \bm{u}
    &= \sum_{k=1}^{\nmin{\tau}} \sum_{l=1}^{\nmin{\tau}}
    \oint_{\Gamma_{-}} \left( u_{k} \fnc{P}_{k}(x,y) \right) \left( u_{l} \fnc{P}_{l}(x,y) \right) \left( \beta_{x} n_{x} + \beta_{y} n_{y} \right) \mr{d}\Gamma \\
    &= \oint_{\Gamma_{-}} \left[\fnc{U}(x,y)\right]^{2} \left( \beta_{x} n_{x} + \beta_{y} n_{y} \right) \mr{d}\Gamma
  \end{align*}
  where $\fnc{U}(x,y) \equiv \sum_{k=1}^{\nmin{\tau}} u_{k} \fnc{P}_{k}(x,y)$.  The
  integrand in the above is the product of a squared polynomial function and the
  non-positive quantity $(\beta_{x} n_{x} + \beta_{y} n_{y}) \leq 0$ $\forall\;
  (x,y) \in \Gamma_{-}$.  Thus the desired result follows.
\qquad\end{proof}

We now turn to the main result of this appendix:
\begin{theorem}\label{thrm:E_decomp}
  Suppose $\mat{F}_{x} = \mat{0}$ and $\mat{F}_{y} = \mat{0}$ in the definitions
  of $\mat{E}_{x}$ and $\mat{E}_{y}$.  Then, for any $\beta_{x},\beta_{y} \in
  \mathbb{R}$, the matrix $\mat{E} \equiv \beta_{x} \Ex + \beta_{y} \Ey$ can be
  decomposed into $\mat{E} = \mat{E}_{+} + \mat{E}_{-}$ where $\mat{E}_{+}$ is
  positive semi-definite, $\mat{E}_{-}$ is negative semi-definite, and
  $\mat{E}_{\pm}$ satisfy the accuracy conditions
  \begin{equation}
    \pk\Tr \mat{E}_{\pm} \pM = 
    \displaystyle\oint_{\Gamma_{\pm}} \fpk \fpm \left( \beta_{x} n_{x} +
    \beta_{y} n_{y} \right) \mr{d}\Gamma, \qquad
    \forall\;k,m \in \{ 1,2,\dots,\nmin{\tau} \}.
    \label{eq:E_accuracy}
  \end{equation}
\end{theorem}

\begin{proof}
  Consider the decomposition
  \begin{equation*}
    \mat{E} = \mat{E}_{+} + \mat{E}_{-} = 
    \tV\invTr 
    \begin{bmatrix} \tilde{\mat{E}}_{+} & \left(\mat{F}_{+}\right)\Tr \\
      \mat{F}_{+} & \mat{G}_{+}
    \end{bmatrix} \tV^{-1}
    +
    \tV\invTr 
    \begin{bmatrix} \tilde{\mat{E}}_{-} & \left(\mat{F}_{-}\right)\Tr \\
      \mat{F}_{-} & \mat{G}_{-},
    \end{bmatrix} \tV^{-1},
  \end{equation*}
  where $\tilde{\mat{E}}_{\pm}$ are defined above, and the pairs
  $(\mat{F}_{+},\mat{F}_{-})$ and $(\mat{G}_{+},\mat{G}_{-})$ are
  yet-to-be-determined decompositions of $\beta_{x} \mat{F}_{x} + \beta_{y}
  \mat{F}_{y}$ and $\beta_{x} \mat{G}_{x} + \beta_{y} \mat{G}_{y}$,
  respectively; note that the above decomposition of $\mat{E}$ is distinct from
  its definition, which involves $\mat{E}_{x}$ and $\mat{E}_{y}$.

  The accuracy of the matrices $\mat{E}_{+}$ and $\mat{E}_{-}$, as defined
  above, can be established using the same approach used in the proof of
  Theorem~\ref{ExistDiagH}.  Therefore, we focus on showing that $\mat{E}_{-}$
  can be made negative semi-definite; again, an analogous proof can be used to
  show $\mat{E}_{+}$ is positive semi-definite.

  To prove that $\mat{E}_{-}$ is negative semi-definite, it suffices to show that
  the matrix
  \begin{equation*}
    \begin{bmatrix} \tilde{\mat{E}}_{-} & \left(\mat{F}_{-}\right)\Tr \\
      \mat{F}_{-} & \mat{G}_{-}
    \end{bmatrix} 
  \end{equation*}
  \ignore{
    \begin{equation*}
      \equiv
      \beta_{x}
      \begin{bmatrix}
        \tEx_{-}& \left(\mat{F}_{x}_{-}\right)\Tr \\
        \mat{F}_{x}_{-} & \mat{G}_{x}_{-}  
      \end{bmatrix}
      +
      \beta_{y}
      \begin{bmatrix}
        \tEy_{-}& \left(\mat{F}_{y}_{-}\right)\Tr \\
        \mat{F}_{y}_{-} & \mat{G}_{y}_{-}  
      \end{bmatrix}
    \end{equation*}
  } is negative semi-definite.  This will be the case if we can ensure
  $\mat{G}_{-}$ is negative-definite and the corresponding Schur complement,
  \begin{equation*}
    \mat{S}_{-} \equiv \tilde{\mat{E}}_{-} - \left(\mat{F}_{-}\right)\Tr \left(
    \mat{G}_{-} \right)^{-1} \mat{F}_{-},
  \end{equation*}
  is negative semi-definite; see, for example, \cite[Appendix A.5.5]{boyd:2004}.

  We first tackle the definiteness of $\mat{G}_{-}$.  Recall that $\mat{G}_{x}$
  and $\mat{G}_{y}$ are symmetric but otherwise arbitrary.  Therefore, the
  matrix $\mat{G} = \beta_{x} \mat{G}_{x} + \beta_{y} \mat{G}_{y}$ has the
  eigendecomposition $\mat{G} = \mat{R} \Lambda \mat{R}\Tr$, where $\mat{R}$
  holds the eigenvectors and the diagonal matrix $\Lambda$ holds the
  eigenvalues.  For any set of eigenvalues, we can construct the nonunique
  decomposition
  \begin{equation*}
    \mat{G} = \mat{R} \Lambda_{+} \mat{R}\Tr + \mat{R} \Lambda_{-} \mat{R}\Tr,
  \end{equation*}
  such that $\Lambda_{+}$ is diagonal positive-definite and $\Lambda_{-}$ is
  diagonal negative-definite; note that any zero eigenvalue in $\Lambda$ can be
  decomposed as $c - c$ for arbitrary $c > 0$.  Equating $\mat{G}_{-}$ with
  $\mat{R} \Lambda_{-} \mat{R}\Tr$ we have that $\mat{G}_{-}$ is symmetric
  negative-definite and therefore invertible.

  Finally, we need to show that $\mat{S}_{-}$ is negative semi-definite.  From
  Lemma~\ref{lem:tilde_E_def}, we have that $\tilde{\mat{E}}_{-}$ is negative
  semi-definite.  Thus, showing $\mat{S}_{-} \preceq \mat{0}$ is equivalent to
  showing\footnote{The notation $A \preceq B$ means $A - B$ is negative
    semi-definite}
  \begin{equation*}
    \tilde{\mat{E}}_{-} \preceq \left(\mat{F}_{-}\right)\Tr \left( \mat{G}_{-}
    \right)^{-1} \mat{F}_{-},
  \end{equation*}
  This statement is true provided the entries in $\mat{F}_{-}$ are sufficiently
  small, which is certainly the case under the assumption that $\mat{F}_{x} =
  \mat{F}_{y} = \mat{0}$.  This concludes the proof.  \qquad\end{proof}

\begin{remark}
  In general, the assumption that $\mat{F}_{x} = \mat{F}_{y} = \mat{0}$ is
  stronger than necessary, since only $\mat{S}_{-} \preceq 0$ is required;
  however, it is not clear how to weaken this assumption when $\beta_{x}$ and
  $\beta_{y}$ are not known a priori.
\end{remark}

\begin{remark}
  The matrices $\mat{E}_{x}$ and $\mat{E}_{y}$ constructed for the simplex
  operators satisfy the conditions of the Theorem~\ref{thrm:E_decomp}, \ie
  $\mat{F}_{x} = \mat{F}_{y} = \mat{0}$.
\end{remark}


\bibliographystyle{siam}
\bibliography{references}

\end{document}